\documentclass[12pt,a4paper,leqno]{amsart}
\usepackage{amsfonts,amssymb,mathrsfs,bm,enumerate,color,graphicx,hyperref}
\usepackage[round]{natbib}

\numberwithin{equation}{section}
\theoremstyle{plain}
 \newtheorem{thm}{Theorem}[section]
 \newtheorem{lem}[thm]{Lemma}
 \newtheorem{cor}[thm]{Corollary}
 \newtheorem{prop}[thm]{Proposition}
 
\theoremstyle{definition}
 \newtheorem{defn}[thm]{Definition}
 
 \newtheorem{rem}[thm]{Remark}

\newcommand{\al}{\alpha}
\newcommand{\bt}{\beta}

\newcommand{\Gm}{\Gamma}

\newcommand{\ld}{\lambda}

\newcommand{\q}{\quad}

\newcommand{\wh}{\widehat}
\newcommand{\wt}{\widetilde}
\newcommand{\la}{\langle}
\newcommand{\ra}{\rangle}

\newcommand{\B}{\mathcal{B}}

\newcommand{\R}{\mathbb{R}}
\newcommand{\N}{\mathbb{N}}
\newcommand{\Z}{\mathbb{Z}}

\newcommand{\rd}{{\mathbb R^d}}
\newcommand{\law}{\mathcal L}
\newcommand{\sek}{\int_0^{\infty}}

\def\1{\scalebox{0.94}{$1$}\hspace{-0.36em}1}

\allowdisplaybreaks[4]
\begin{document}
\setlength{\baselineskip}{18pt}
\setlength{\parindent}{1.8pc}

\title{Nested subclasses of the class of $\alpha$-selfdecomposable distributions}
\subjclass[2000]{60E07, 60G51, 60F05}
\keywords{selfdecomposable distribution, $\alpha$-selfdecomposable distribution, nested subclass,
limit theorem, mapping of infinitely divisible distribution}
\maketitle

\begin{center}
{{\sc Makoto Maejima\footnote{maejima@math.keio.ac.jp} 
and Yohei Ueda\footnote{ueda@2008.jukuin.keio.ac.jp}}\\
\vskip 3mm
Department of Mathematics, Keio University\\
3-14-1, Hiyoshi, Kohoku-ku\\
Yokohama 223-8522, Japan}
\end{center}

\begin{abstract}
A probability distribution $\mu$ on $\rd$ is selfdecomposable if its characteristic function $\wh\mu(z), z\in\rd$,
satisfies that for any $b>1$, there exists an infinitely divisible distribution $\rho_b$ satisfying
$\wh\mu(z) = \wh \mu (b^{-1}z)\wh\rho_b(z)$.
This concept has been generalized to the concept of $\al$-selfdecomposability by many authors in the following way.
Let $\al\in\R$.
An infinitely divisible distribution $\mu$ on $\rd$ is $\al$-selfdecomposable, if 
 for any $b>1$, there exists an infinitely divisible distribution $\rho_b$ satisfying
$\wh\mu(z) = \wh \mu (b^{-1}z)^{b^{\al}}\wh\rho_b(z)$.
By denoting the class of all $\al$-selfdecomposable distributions on $\rd$ by $L^{\la\al\ra}(\rd)$, we define in this paper
a sequence of nested subclasses of $L^{\la\al\ra}(\rd)$, 
and investigate several properties of them by two ways.
One is by using limit theorems and the other is by using mappings of infinitely divisible distributions.
\end{abstract}
\section{Introduction}
Let $\mathscr P(\rd)$ and $I(\rd)$ be the class of all probability distributions on $\rd$ and the class of all infinitely 
divisible distributions on $\rd$, respectively, 
and let
$I_{\log^{m}}(\rd)= \{ \mu\in I(\rd)\colon \int_{\rd}(\log ^+|x|)^{m}\mu(dx)<\infty\}$ for $m\in\N$ and 
$I_{\log}(\rd):=I_{\log^1}(\rd)$, where $|x|$ is the Euclidean norm of $x\in\rd$ and $\log^+|x|= (\log|x|)\vee 0$.
The terminology of $\alpha$-selfdecomposability was introduced in \citet{MaejimaUeda2009b}.
This is a generalization of selfdecomposability.
Here $\mu\in\mathscr P(\rd)$ is said to be selfdecomposable if for each $b>1$ there exists $\rho_b\in \mathscr P(\rd)$ 
satisfying $\widehat{\mu}(z)=\widehat{\mu}(b^{-1}z)\widehat{\rho}_b(z),z\in\rd$,
where $\wh\mu(z),z\in\rd$, stands for the characteristic function of $\mu\in\mathscr P(\rd)$.
These $\rho_b$ automatically belong to $I(\rd)$. 
We denote the totality of selfdecomposable distributions on $\rd$ by $L(\rd)$.
It is well known that $L(\rd)\subset I(\rd)$.
Our generalization of selfdecomposability is as follows.
\begin{defn}[\citet{MaejimaUeda2009b}]\label{alpha-selfdecomposable}
Let $\alpha\in\R$. 
We say that $\mu\in I(\rd)$ is \emph{$\alpha$-selfdecomposable}, 
if for any $b>1$, there exists $\rho_b\in I(\rd)$ satisfying
\begin{equation}\label{def_characteristic_function}
\widehat{\mu}(z)=\widehat{\mu}(b^{-1}z)^{b^\alpha}\widehat{\rho}_b(z),\quad z\in\rd.
\end{equation}
We denote the totality of $\alpha$-selfdecomposable distributions on $\rd$ 
by $L^{\la\alpha\ra}(\rd)$.
\end{defn}
Note that $L^{\la0\ra}(\rd)=L(\rd)$.
And $L^{\la-1\ra}(\rd)$ is the class of all $s$-selfdecomposable distributions on $\rd$, which is sometimes 
written as $U(\rd)$ and was studied deeply by Jurek, (see, e.g., \citet{Jurek1981,Jurek1985,Jurek2004} or 
\citet{IksanovJurekSchreiber2004}).
Also, the classes $L^{\la\al\ra}(\rd), \alpha\in\R,$ and similar ones were already studied by
several authors.
\citet{Jurek1988,Jurek1989,Jurek1992}, and \citet{JurekSchreiber1992} studied
the classes $\mathscr U_\beta(Q), \beta\in\R$, of distributions on a real separable Banach 
space $E$, where $Q$ is a linear operator on $E$ with certain properties. 
These classes are equal to $L^{\la\alpha\ra}(\rd)$ if $\beta=-\alpha$, $E=\rd$ and $Q$ is the identity operator.
As to these classes, they studied the decomposability and stochastic integral characterizations,
although some results are only for the case that $Q$ is the identity operator.
However, since, for $0<\alpha< 2$, $L^{\la\alpha\ra}(\rd)$ contains all $\alpha$-stable distributions and 
any $\mu\in L^{\la \al\ra}(\rd)$ 
belongs to the normal domain of attraction of some $\alpha$-stable distribution,
we adopt the parametrization in Definition \ref{alpha-selfdecomposable}.
For details on this history, see \citet{MaejimaUeda2009b}.

$L(\rd)$ is characterized by, for example, radial components of L\'evy measures, a stochastic integral representation,
and
the relation to Ornstein-Uhlenbeck type processes, (see, e.g., \citet{Sato's_book2003}).
By \citet{MaejimaUeda2009b} and others, these characterizations of $L(\rd)$ were generalized to
$L^{\la\alpha\ra}(\rd)$.

As to nested subclasses of $L(\rd)$, the following are known, (see, e.g., \linebreak
\citet{Sato's_book2003}).
Define nested subclasses $L_m(\rd),m\in\Z_+$ of $L(\rd)$ in the following way:
$\mu \in L_m(\rd)$ if and only if for each $b>1$, there exists $\rho_b\in L_{m-1}(\rd)$ such that
$\wh \mu (z) = \wh\mu (b^{-1}z)\wh\rho_b (z)$,
where $L_0(\rd):=L(\rd)$.
Since, by definition, $L_m(\rd)\supset L_{m+1}(\rd)$, these are called nested subclasses.
Besides, we introduce an operation $\mathfrak Q(\cdot)$ in the following way:
Let $H\subset \mathscr P(\rd)$.
We say that $\mu\in \mathscr P(\rd)$ belongs to $\mathfrak Q(H)$ if there exist sequences 
$\{X_n\}$ of $\rd$-valued independent random variables, $\{a_n\}\subset(0,\infty)$, 
and $\{c_n\}\subset\rd$
such that
$\{\law(X_n),n\in\N\}\subset H$, $\{a_n^{-1}X_j,1\leq j\leq n;n\in\N\}$ is infinitesimal, and
$$
\law\left(a_n^{-1}\sum_{j=1}^nX_j+c_n\right)\to\mu\quad\text{as }n\to\infty,
$$
where $\law (X)$ means the law of a random variable $X$.
Then it is known that $L_0(\rd)=\mathfrak Q(\mathscr P(\rd))=\mathfrak Q(I(\rd))$ and 
$L_m(\rd)=\mathfrak Q(L_{m-1}(\rd))$ for $m\in\N$ so that
$L_m(\rd)=\mathfrak Q^{m+1}(\mathscr P(\rd))=\mathfrak Q^{m+1}(I(\rd))$ for $m\in\Z_+$, 
where $\mathfrak Q^{m+1}(\cdot)$ denotes the $m+1$ times iteration of the $\mathfrak Q(\cdot)$-operation.
On the other hand, if we define a mapping $\Phi$ by
$$
\Phi(\mu)=\law\left(\int_0^\infty e^{-t}dX_t^{(\mu)}\right),\q \mu\in I_{\log}(\rd),
$$
where $\{X_t^{(\mu)},t\geq 0\}$ is a L\'evy process on
 $\rd$ with $\mu\in I(\rd)$ as its distribution at time 1,
then it is known that for $m\in\Z_+$, $L_m(\rd)$ is realized as the range of 
the $m+1$ times composition of $\Phi$, namely,
$\mathfrak R(\Phi^{m+1})=L_m(\rd)$,
where the domain of $\Phi^{m+1}$ is $I_{\log^{m+1}}(\rd)$.
Furthermore, the limit $L_\infty(\rd):=\lim_{m\to\infty}L_m(\rd)=\bigcap_{m=0}^\infty L_m(\rd)$ is known to be equal to  
$\overline{S(\rd)}$, 
which is the closure under convolution and weak convergence, of the class of all stable distributions.
Namely,
$$
\lim_{m\to\infty}\mathfrak Q^{m+1}(\mathscr P(\rd))=\lim_{m\to\infty}\mathfrak Q^{m+1}(I(\rd))=
\lim_{m\to\infty}\mathfrak R(\Phi^{m+1})=L_{\infty}(\rd)=\overline{S(\rd)}.
$$

The following was already done as to nested subclasses of $L^{\la\alpha\ra}(\rd),\alpha\in\R$.
\citet{Jurek2004} studied nested subclasses of $L^{\la-1\ra}(\rd)$, 
\citet{MaejimaSato2009} found the limit of the nested subclasses of $L^{\la\alpha\ra}(\rd),-1\leq \alpha<0,$ defined by mappings,
and \citet{MaejimaMatsuiSuzuki} investigated nested subclasses of $L^{\la\alpha\ra}(\rd),\alpha<2,$ in terms of mappings.
However, the study on nested subclasses of $L^{\la\alpha\ra}(\rd),\alpha\in\R,$ in terms of limit theorems and mappings is not completed yet
and 
the purpose of this paper is to do it.

\citet{MaejimaSato2009} proved that the limits of several nested classes defined by stochastic integral mappings
are identical with $\overline{S(\rd)}$. 
Then a natural question arose.
Can we find mappings by which, as the limit of iteration, we get a larger or a smaller class than 
$\overline{S(\rd)}$?
\citet{Sato20072008} constructed mappings producing a class smaller than $\overline{S(\rd)}$
and \citet{MaejimaUeda2009} found mappings which produce a larger class than $\overline{S(\rd)}$.
In Theorems \ref{R(Phi_{alpha}^infty)}, we will see that stochastic integral mappings associated with classes $L^{\la\alpha\ra}(\rd),\alpha\in(0,2),$
make smaller classes than $\overline{S(\rd)}$ as the limits of the ranges of their iteration,
which is the same iterated limit as that of Sato's mappings above.
Also,
in Corollary \ref{Phi_alpha^infty(H)}, we see a result about nested classes of $L^{\la\alpha\ra}(\rd)$ based on $H\subset I(\rd)$ with certain properties instead of $I(\rd)$,
which enable us to find the iterated limit of some other stochastic integral mappings, (see Remark \ref{application} and \citet{MaejimaUeda2009e}).

Organization of this paper is as follows.
In Section 2, we explain necessary notation and give some preliminaries.
In Section 3, we study nested subclasses of $L^{\la\alpha\ra}(\rd)$ in terms of a limit theorem.
In Section 4, we investigate nested subclasses of $L^{\la\alpha\ra}(\rd)$ in terms of a mapping of
infinitely divisible distributions,
by using the results in Section 3.
In Section 5, a supplementary remark is mentioned.
\vskip 10mm
\section{Notation and preliminaries}
In this section, we explain necessary notation and give some preliminaries.

Throughout this paper, we use the L\'evy-Khintchine representation of the characteristic function of $\mu\in I(\R^d)$ 
in the following form:
$$
\widehat{\mu}(z)=\exp\left\{-\frac 12 \la z,Az\ra +i\la \gamma,z\ra +\int_\rd 
\left(e^{i\la z,x\ra}-1-\frac{i\la z,x\ra}{1+|x|^2}\right)\nu(dx)\right\},\quad z\in\rd,
$$
where $\la\cdot,\cdot\ra$ is Euclidean inner product on $\rd$ respectively, $A$ is a 
nonnegative-definite symmetric $d\times d$ 
matrix, $\gamma\in\rd$, 
and $\nu$ is a measure satisfying $\nu(\{0\})=0$ and $\int_\rd(|x|^2\wedge 1)\nu(dx)<\infty$. 
$\nu$ is called the L\'evy measure of $\mu\in I(\rd)$.
We also call $(A,\nu,\gamma)$ the L\'evy-Khintchine triplet of $\mu$ and we write $\mu=\mu_{(A,\nu,\gamma)}$ 
when we want to emphasize its L\'evy-Khintchine triplet.
$C_{\mu}(z),z\in\rd$, denotes the cumulant function of $\mu\in I(\rd)$, that is, $C_{\mu}(z)$ is the unique continuous function
satisfying $\wh\mu(z) = e^{C_{\mu}(z)}$ and $C_{\mu}(0)=0$.
For $\mu\in I(\rd)$ and $t>0$, we call the distribution with characteristic function $\wh\mu(z)^t:=e^{tC_\mu(z)}$ the 
$t$-th convolution of $\mu$ and denote it by $\mu^t$.

A set $H\subset\mathscr P(\rd)$ is said to be closed under type equivalence if $\law(X)\in H$ implies $\law(aX+c)\in H$ 
for $a>0$,
 and $c\in\rd$. $H\subset I(\rd)$ is called completely closed in the strong sense (abbreviated as {c.c.s.s.})
if $H$ is closed 
under convolution, weak convergence, type equivalence, and $t$-th convolution for any $t>0$.
Note that $I(\rd)$ and $L(\rd)$ are c.c.s.s., but $S(\rd)$ is not.

$\B_0(\rd)$ denotes the totality of $B\in\B(\rd)$ satisfying $\inf_{x\in B}|x|>0$.
Let $S=\{x\in\rd\colon|x|=1\}$ and we write, for $E\in\B((0,\infty))$ and $C\in\B(S)$, 
$EC:=\{x\in\rd\setminus\{0\}\colon|x|\in E\text{ and }x/|x|\in C\}$.

We also use stochastic integrals with respect to L\'evy processes.
Stochastic integrals with respect to L\'evy processes $\{X_t, t\ge 0\}$ of nonrandom measurable 
functions $f\colon [0,\infty)\to\R$, which are $\int_0^tf(s)dX_s ,t\in[0,\infty)$, are deeply studied in \citet{Sato2004,Sato2006a},
and his way of defining a stochastic integral with respect to a L\'evy process is to define a stochastic integral
based on the $\rd$-valued independently scattered random measure induced by a L\'evy process on $\rd$.
The improper stochastic integral $\sek f(s)dX_s$ is defined as the limit in probability of 
$\int_0^tf(s)dX_s$
as $t\to\infty$ whenever the limit exists.

Using stochastic integrals with respect to L\'evy processes, we can define a mapping
\begin{equation}\label{mapping}
\Phi_f(\mu) = \law \left ( \sek f(t)dX_t^{(\mu)}\right ) ,\quad \mu\in\mathfrak D(\Phi_f)\subset I(\rd),
\end{equation}
for a nonrandom measurable function $f\colon [0,\infty)\rightarrow \R$,
where $\mathfrak D (\Phi_f)$ is the domain of a mapping 
$\Phi_f$ that is
the class of $\mu\in I(\rd)$ for which $\sek f(t)dX_t^{(\mu)}$ is definable in the sense above. 
When we consider the composition of two mappings $\Phi_f$ and $\Phi_g$, denoted by $\Phi_g\circ\Phi_f$,
the domain of $\Phi_g\circ\Phi_f$ is $\mathfrak D (\Phi_g\circ \Phi_f) = \{ \mu\in I(\rd)\colon\mu\in \mathfrak D (\Phi_f)
\text{ and }\Phi_f(\mu)\in \mathfrak D (\Phi_g)\}$.
Also, for a mapping $\Phi_f$ and $m\in\N$, we denote by $\Phi_f^m$ the $m$ times composition of $\Phi_f$ itself.
Once we define such a mapping, we can characterize a subclass of $I(\rd)$ as the range of $\Phi_f$, 
$\mathfrak R (\Phi_f):=\Phi_f(\mathfrak D (\Phi_f))$.
See also \citet{Sato2006}.
\vskip 10mm
\section{Nested subclasses of the class of $\alpha$-selfdecomposable distributions defined by limit theorems
and their characterizations in terms of L\'evy measures}
We start this section with the following definition, which defines a subclass of $I(\rd)$ through a limit theorem.
\begin{defn}\label{limit_theorem}
Let $\alpha\in\R$ and $H\subset I(\rd)$.
$\mu\in \mathscr P(\rd)$ is said to belong to the class $\mathfrak Q_\alpha(H)$ if 
there exist a sequence $\{\mu_j,j\in\N\}\subset I(\rd)$ satisfying $\{\mu_j,j\geq j_0\}\subset H$ for some 
$j_0\in\N$, $a_n>0,\,\uparrow \infty $ satisfying $a_{n+1}/a_n\rightarrow 1$, $c_n\in\rd$, and $p_n>0$ 
satisfying $p_n/ a_n^\alpha\to 1$ such that
\begin{equation}\label{limit}
\lim_{n\rightarrow \infty}\prod_{j=1}^n\wh\mu_j(a_n^{-1}z)^{p_n}e^{i\la c_n,z\ra}= \wh\mu(z),\quad\text{for }z\in\rd.
\end{equation}
\end{defn}
\begin{rem}
In Definition \ref{limit_theorem}, we assume $H$ to be a subclass of $I(\rd)$ 
because we need the $t$-th convolution 
of its elements for $t>0$. 
Due to this assumption, we do not need the infinitesimal condition, as \citet{Jurek2004} remarked. 
Then, Definition \ref{limit_theorem} is similar to the limit theorem 
characterizing the class of selfdecomposable distributions $L(\rd)$.
\end{rem}
The following is immediately obtained by definition.
\begin{lem}\label{Inclusion}
Let $\alpha\in\R$. If $H_1\subset H_2\subset I(\rd)$, then $\mathfrak Q_\alpha(H_1)\subset \mathfrak Q_\alpha(H_2)$.
\end{lem}
We can characterize the classes $\mathfrak Q_\alpha(H)$ by the  decomposability and
$L^{\la\al\ra}(\rd)$ by the $\mathfrak Q_\alpha(\cdot)$-operation as follows.
\begin{thm}\label{characterization_c.f.}
Let $\alpha\in\R$ and let $H\subset I(\rd)$ be c.c.s.s.
\begin{enumerate}[(i)]
\item $\mu\in \mathfrak Q_\alpha(H)$ if and only if $\mu\in I(\rd)$ and for each $b>1$ there exists $\rho_b\in H$ satisfying
\eqref{def_characteristic_function}.
\item $\mathfrak Q_\alpha(I(\rd)) = L^{\la \al \ra}(\rd)$.
\end{enumerate}
\end{thm}
\begin{proof}
(i) We first show the ``if'' part.
Let $\mu\in I(\rd)$ and for each $b>1$ there exists $\rho_b\in H$ satisfying
\eqref{def_characteristic_function}. Then, it suffices to set $\mu_1:=\mu\in I(\rd)$, $\wh\mu_j(z):=
\wh\rho_{j/(j-1)}(jz)^{j^{-\alpha}}$ for $j\geq 2$, $a_n:=n$, $p_n:=n^\alpha$, and $c_n:=0$. 
Indeed, $\{\mu_j,j\geq 2\}\subset H$ since $H$ is c.c.s.s., and
for all $n\ge 2$,
\begin{align*}
\prod_{j=1}^n\wh\mu_j(a_n^{-1}z)^{p_n}e^{i\la c_n,z\ra}
&=\wh\mu\left(\frac{1}{n}z\right)^{n^{\alpha}}\prod_{j=2}^n\wh\rho_{j/(j-1)}
\left(\frac{j}{n}z\right)^{\left(\frac{n}{j}\right)^{\alpha}}\\
&=\wh\mu\left(\frac{1}{n}z\right)^{n^{\alpha}}\prod_{j=2}^n
\frac
{\wh\mu\left(\frac{j}{n}z\right)^{\left(\frac{n}{j}\right)^{\alpha}}}
{\wh\mu\left(\frac{j-1}{n}z\right)^{\left(\frac{n}{j-1}\right)^{\alpha}}}
=\wh\mu(z),
\end{align*}
implying \eqref{limit}.

We next show the ``only if'' part.
For any $b>1$, we can take $n_l,m_l\in\N$ diverging to $\infty$ such that $m_l<n_l$ and 
$a_{n_l}a_{m_l}^{-1}\rightarrow b$ as $l \to \infty$.
This is possible, due to the argument in the proof of Theorem 15.3 (i) of \citet{Sato's_book1999}. 
Then,
\begin{align*}
\prod_{j=1}^{n_l}\wh\mu_j\left(a_{n_l}^{-1}z\right)^{p_{n_l}}e^{i\la c_{n_l},z\ra}
=&\left\{\prod_{j=1}^{m_l}\wh\mu_j\left(a_{m_l}^{-1}(a_{m_l}a_{n_l}^{-1}z)\right)^{p_{m_l}}
e^{i\left\la c_{m_l},a_{m_l}a_{n_l}^{-1}z\right\ra}\right\}^{p_{n_l}p_{m_l}^{-1}}\\
&\quad\times \prod_{j=m_l+1}^{n_l}\wh\mu_j\left(a_{n_l}^{-1}z\right)^{p_{n_l}}
e^{i\left\la c_{n_l}-c_{m_l}a_{m_l}a_{n_l}^{-1}p_{n_l}p_{m_l}^{-1},z\right\ra},
\end{align*}
where the left-hand side and the first term of right-hand side tend to $\wh\mu(z)$ and 
$\wh\mu(b^{-1}z)^{b^\alpha}$ as $l\to\infty$, respectively, by virtue of the uniform convergence of 
the characteristic functions. 
Since $\wh\mu(z)$ is the limit of the sequence of infinitely divisible distributions, 
$\mu$ is also infinitely divisible 
and thus $\wh\mu(b^{-1}z)^{b^\alpha}\neq 0$ for all $z\in\rd$.
The second term of the right-hand side converges to $\wh\mu(z)/\wh\mu(b^{-1}z)^{b^\alpha}$ 
which is continuous at $z=0$ and therefore the characteristic function of some probability measure $\rho_b$. 
Then, \eqref{def_characteristic_function} holds. 
Furthermore, since $\{\mu_j,j\geq j_0\}\subset H$ and $H$ is c.c.s.s., we have $\rho_b\in H$.

(ii) This is an immediate consequence of the part (i) that we
 have just shown and the definition of $L^{\la\al\ra}(\rd)$.
\end{proof}
The following holds from Theorem \ref{characterization_c.f.}.
\begin{cor}\label{alpha>=2}
Let $H (\neq \emptyset) \subset I(\rd)$ be c.c.s.s. 
Then, $\mathfrak Q_2(H)$ is the class of all Gaussian distributions on $\rd$, and for $\alpha>2$, $\mathfrak Q_\alpha(H)$ 
is the class of all $\delta$-distributions on $\rd$.
\end{cor}
\begin{proof}
We first prove that $\mathfrak Q_\alpha(H)$ 
includes the class of all Gaussian distributions if $\alpha=2$, and all $\delta$-distributions if $\alpha>2$.
Indeed, if $H(\neq \emptyset)\subset I(\rd)$ is c.c.s.s., then
$\mu\in H$ exists and for all $\gamma\in\rd$, $\delta_\gamma=\lim_{n\to\infty}\mu^{1/n}*\delta_\gamma\in H$ 
since $H$ is c.c.s.s.
If $\mu$ is Gaussian, then for each $b>1$, there is $c_b\in\rd$ satisfying $\wh\mu(z)=\wh\mu(b^{-1}z)^{b^2}e^{i\la c_b,z\ra}$. 
Also, if $\alpha>2$ and $\mu$ is a $\delta$-distribution, then for each $b>1$, there is $c_b\in\rd$ 
satisfying $\wh\mu(z)=\wh\mu(b^{-1}z)^{b^\alpha}e^{i\la c_b,z\ra}$. Noting that $\delta_{c_b}\in H$, we have the assertion.

We next show that $\mathfrak Q_\alpha(H)$ 
is included in the class of all Gaussian distributions if $\alpha=2$, and all $\delta$-distributions if $\alpha>2$.
By Lemma \ref{Inclusion} and Theorem \ref{characterization_c.f.} (ii), we have $\mathfrak Q_\alpha(H)\subset L^{\la\alpha\ra}(\rd)$. Note that $L^{\la\alpha\ra}(\rd)$ is equal to the class of all Gaussian distributions if $\alpha=2$, and all $\delta$-distributions if $\alpha>2$, (see \citet{MaejimaUeda2009b}).
\end{proof}
For $0<\beta\leq 2$, $S_\beta(\rd)$ stands for the totality of $\beta$-stable distributions on $\rd$.
Let $S(\rd):=\bigcup_{\beta\in(0,2]}S_\beta(\rd)$.
\begin{cor}\label{0<alpha<=2}
Let $0<\alpha\leq 2$. Then, $\mathfrak Q_\alpha\left(\{\delta_\gamma\colon\gamma\in\rd\}\right)=S_\alpha(\rd)$.
\end{cor}
\begin{proof}
Note that $\mu\in S_\alpha(\rd)$ if and only if for each $b>1$ there exists $c_b\in\rd$ satisfying $\wh\mu(z)=\wh\mu(b^{-1}z)^{b^\alpha}e^{i\la c_b,z\ra}$.
Then, Theorem \ref{characterization_c.f.} (i) implies the statement.
\end{proof}
For $\alpha<2$, let
$$
\mathcal C_\alpha(\rd)
:=\left\{\mu=\mu_{(A,\nu,\gamma)}\in I(\rd)\colon \lim_{r\rightarrow \infty }
r^\alpha\int_{|x|>r}\nu(dx)=0\right\}.
$$
Note that $\mathcal C_\alpha(\rd)=I(\rd)$ if $\alpha\leq 0$.
If $\alpha<2$, $\mu\in\mathcal C_\alpha(\rd)$ and $H$ is c.c.s.s., then, $\mu\in\mathfrak Q_\alpha(H)$ can be characterized by a limit theorem slightly different from Definition \ref{limit_theorem} as follows.
\begin{thm}\label{limit_theorem2}
Let $\alpha<2$ and let $H\subset I(\rd)$ be c.c.s.s.
Assume $\mu\in\mathcal C_\alpha(\rd)$.
Then,
$\mu\in \mathfrak Q_\alpha(H)$ if and only if
there exist a sequence $\{\mu_j,j\in\N\}\subset H$, $a_n>0,\,\uparrow \infty $ satisfying $a_{n+1}/a_n\rightarrow 1$, $c_n\in\rd$, and $p_n>0$ 
satisfying $p_n/ a_n^\alpha\to 1$ such that
$$
\lim_{n\rightarrow \infty}\prod_{j=1}^n\wh\mu_j(a_n^{-1}z)^{p_n}e^{i\la c_n,z\ra}= \wh\mu(z),\quad\text{for }z\in\rd.
$$
\end{thm}
\begin{proof}
The ``if" part is trivial by Definition \ref{limit_theorem}.

Let us prove the ``only if" part.
If $\mu=\mu_{(A,\nu,\gamma)}\in \mathfrak Q_\alpha(H)$, then for each $b>1$, there exists $\rho_b\in H$ satisfying 
\eqref{def_characteristic_function} by virtue of Theorem \ref{characterization_c.f.} (i). 
Then, it suffices to set $\wh\mu_j(z):=\wh\rho_{(j+1)/j}((j+1)z)^{(j+1)^{-\alpha}}$, $a_n:=n$, $p_n:=n^\alpha$, $c_n:=0$ 
if $\alpha\leq 0$, and $c_n:=n^{\alpha-1}\gamma+n^{\alpha}\int_\rd x\left\{(1+|x|^2)^{-1}-(1+|nx|^2)^{-1}\right\}\nu(n\,dx)$ 
if $0<\alpha<2$. Indeed, $\{\mu_j,j\in\N\}\subset H$ since $H$ is c.c.s.s. and
$$
\prod_{j=1}^n\wh\mu_j(a_n^{-1}z)^{p_n}=\prod_{j=1}^n\wh\rho_{(j+1)/j}\left(\frac{j+1}{n}z\right)^{\left(\frac{n}{j+1}\right)^{\alpha}}
=\prod_{j=1}^n
\frac
{\wh\mu\left(\frac{j+1}{n}z\right)^{\left(\frac{n}{j+1}\right)^{\alpha}}}
{\wh\mu\left(\frac{j}{n}z\right)^{\left(\frac{n}{j}\right)^{\alpha}}}
=\frac
{\wh\mu\left(\frac{n+1}{n}z\right)^{\left(\frac{n}{n+1}\right)^{\alpha}}}
{\wh\mu\left(\frac{1}{n}z\right)^{n^{\alpha}}}
$$
which tends to $\wh\mu(z)$ as $n\rightarrow \infty$, if $\alpha\leq 0$.
If $0<\alpha<2$, we have
$$
n^{\alpha}C_\mu\left(n^{-1}z\right)-i\la c_n,z\ra=-\frac{1}{2}n^{\alpha-2}\la z,Az\ra+n^\alpha\int_\rd 
\left(e^{i\la z,x\ra}-1-\frac{i\la z,x\ra}{1+|x|^2}\right)\nu(n\,dx).
$$
For any bounded continuous function $f\colon\rd\rightarrow \R$ vanishing on a neighborhood of $0$, it follows that
$$
\lim_{n\rightarrow \infty} n^\alpha\int_\rd f(x)\nu(n\,dx)=0,
$$
since $\mu\in \mathcal C_\alpha(\rd)$. Recalling that $\nu(B)\geq n^{\alpha}\nu(nB)$ for $B\in\B(\rd)$ 
from \eqref{def_characteristic_function}, we have
\begin{align*}
\lim_{\varepsilon\downarrow 0}&\overline{\lim_{n\rightarrow \infty }}\left|n^{\alpha-2}
\la z,Az\ra+n^\alpha\int_{|x|\leq \varepsilon }\la z,x\ra^2\nu(n\,dx)\right|\\
&\leq \lim_{n\rightarrow \infty }n^{\alpha-2}\left|\la z,Az\ra\right|+\lim_{\varepsilon\downarrow 0}
\int_{|x|\leq \varepsilon }\la z,x\ra^2\nu(dx)=0.
\end{align*}
Then, it follows from Theorem 8.7 of \citet{Sato's_book1999} that $\lim_{n\to\infty}\wh\mu
\left(n^{-1}z\right)^{n^{\alpha}}e^{-i\la c_n,z\ra}=1$. Thus 
$$
\prod_{j=1}^n\wh\mu_j(a_n^{-1}z)^{p_n}e^{i\la c_n,z\ra}
=\frac
{\wh\mu\left(\frac{n+1}{n}z\right)^{\left(\frac{n}{n+1}\right)^{\alpha}}}
{\wh\mu\left(\frac{1}{n}z\right)^{n^{\alpha}}e^{-i\la c_n,z\ra}}
\to\wh\mu(z),
$$
as $n\to\infty$.
\end{proof}
Corollaries \ref{alpha>=2}, \ref{0<alpha<=2} and Theorem \ref{limit_theorem2} yield the following.
\begin{cor}\label{Q_alpha(H) H}
\begin{enumerate}[(i)]
\item Let $\alpha\in(-\infty,0]\cup(2,\infty)$. Then, for all c.c.s.s.\,$H\subset I(\rd)$, 
$\mathfrak Q_\alpha(H)\subset H$.
\item Let $\alpha\in(0,2]$. Then, there exists a c.c.s.s.\,$H\subset I(\rd)$ satisfying $\mathfrak Q_\alpha(H)\not\subset H$. 
\item Let $\alpha\in(0,2)$. Then, for all c.c.s.s.\,$H\subset I(\rd)$, 
$\mathfrak Q_\alpha(H)\cap\mathcal C_\alpha(\rd)\subset H$.
\end{enumerate}
\end{cor}
\begin{proof}
(i) If $\alpha\leq 0$, and $H\subset I(\rd)$ is c.c.s.s., then Theorem \ref{limit_theorem2} implies $\mathfrak Q_\alpha(H)\subset H$. 
Let $\alpha> 2$ and let $H\subset I(\rd)$ be c.c.s.s. If $H=\emptyset$, then $\mathfrak Q_\alpha(\emptyset)=\emptyset$. 
Assume $H\neq\emptyset$. Then Corollary \ref{alpha>=2} implies $\mathfrak Q_\alpha(H)=\{\delta_\gamma\colon 
\gamma\in\rd\}\subset H$.

(ii) See Corollary \ref{0<alpha<=2}.

(iii) See Theorem \ref{limit_theorem2}.
\end{proof}
We are ready to define nested subclasses of $L^{\la\al\ra}(\rd)$ by using the $\mathfrak Q_{\al}(\cdot)$-operation.
Let $H\subset I(\rd)$ and $\alpha\in\R$. For $m=0,1,2,\dots,\infty$, we denote the $m$ times iteration of 
$\mathfrak Q_\alpha(\cdot )$ by $\mathfrak Q_\alpha^m(\cdot)$, namely,
$$
\mathfrak Q_\alpha^m(H)=\underbrace{\mathfrak Q_\alpha(\mathfrak Q_\alpha(\cdots (\mathfrak Q_\alpha}_{m}(H))\cdots)),
$$
where $\mathfrak Q_\alpha^0(H)=H$, and $\mathfrak Q_\alpha^\infty(H)=\bigcap_{m=1}^\infty \mathfrak Q_\alpha^m(H)$.
By Corollary \ref{Q_alpha(H) H} (ii), it is not always true that
$\mathfrak Q_\alpha^1(H)\subset \mathfrak Q_\alpha^0(H)(=H)$.
However, it will be seen in Proposition \ref{Q_alpha^m_property} (iii) that if $H\subset I(\rd)$ is c.c.s.s., 
then $\mathfrak Q_\alpha^{m+1}(H)\subset \mathfrak Q_\alpha^m(H), m\in\N$, so that 
$\lim_{m\to\infty}\mathfrak Q_\alpha^m(H)=\bigcap_{m=1}^\infty \mathfrak Q_\alpha^m(H)$,
if we regard $\mathfrak Q_\alpha^m(H)$ as a sequence with $m\in\N$.

For $0<\alpha<2$, let
$$
I_\alpha(\rd):=
\left\{ \mu\in I(\rd)\colon
 \int_{\rd}|x|^\alpha \mu(dx)<\infty\right\}.
$$
We first prepare the following lemma.
\begin{lem}\label{rho_b in I_alpha}
Let $0<\alpha<2$. Suppose $\mu\in L^{\la\alpha\ra}(\rd)$. Then, for all $b>1$, $\rho_b$ in Definition 
\ref{alpha-selfdecomposable} satisfies $\rho_b\in I_\alpha(\rd)$.
\end{lem}
\begin{proof}
Let $b>1$. Denoting the L\'evy measures of $\mu$ and $\rho_b$ by $\nu$ and $\nu_b$, respectively, we have that
$\nu_b(B)=\nu(B)-b^\alpha\nu(bB)$ for $B\in\B_0(\rd)$ by \eqref{def_characteristic_function}.
Then, it follows that
\begin{align*}
\int_{|x|>1}|x|^\alpha\nu_b(dx)
&=\sum_{k=0}^\infty\int_{|x|\in (b^k,\,b^{k+1}]}|x|^\alpha\nu_b(dx)\leq 
\sum_{k=0}^\infty b^{\alpha(k+1)}\nu_b\left((b^k,\,b^{k+1}]S\right)\\
&=\sum_{k=0}^\infty b^{\alpha(k+1)}\left\{\nu\left((b^k,b^{k+1}]S\right)-b^\alpha\nu
\left((b^{k+1},b^{k+2}]S\right)\right\}\\
&=\sum_{k=0}^\infty \left\{b^{\alpha(k+1)}\nu\left((b^k,b^{k+1}]S\right)-b^{\alpha(k+2)}\nu
\left((b^{k+1},b^{k+2}]S\right)\right\}\\
&=\lim_{n\rightarrow \infty }\left\{b^{\alpha}\nu\left((1,b]S\right)-b^{\alpha(n+2)}\nu
\left((b^{n+1},b^{n+2}]S\right)\right\}\\
&\leq b^{\alpha}\nu\left((1,b]S\right)<\infty.
\end{align*}
This implies $\mu\in I_\alpha(\rd)$, due to Corollary 25.8 of \citet{Sato's_book1999}.
\end{proof}
We now prove several properties of $\mathfrak Q_\alpha^m(H)$.
\begin{prop}\label{Q_alpha^m_property}
Let $H\subset I(\rd)$ be c.c.s.s. Then, we have the following.
\begin{enumerate}[(i)]
\item For $\alpha\in\R$ and $m\in\{0,1,2,\dots,\infty\}$, $\mathfrak Q_\alpha^m(H)$ is also c.c.s.s.
\item For $\alpha\in\R$ and $m\in\Z_+$, $\mu\in \mathfrak Q_\alpha^{m+1}(H)$ if and only if $\mu\in I(\rd)$ and 
for each $b>1$ there exists 
$\rho_b\in \mathfrak Q_\alpha^m(H)$ satisfying
\eqref{def_characteristic_function}.
\item Let $\alpha\in\R$. Then, $\mathfrak Q_\alpha^m(H)$ is decreasing in $m\in\N$ 
with respect to set inclusion, namely,
\begin{equation}\label{decreasing in m}
\mathfrak Q_\alpha^{m}(H)\supset \mathfrak Q_\alpha^{m+1}(H)\quad\text{for }m\in \N.
\end{equation}
\item Let $\alpha\in\R$. Then $\mathfrak Q_\alpha^\infty(H)$ is invariant under 
the $\mathfrak Q_\alpha^\infty(\cdot)$-operation, that is, 
$$
\mathfrak Q_\alpha\left(\mathfrak Q_\alpha^\infty(H)\right)=\mathfrak Q_\alpha^\infty(H),
$$
which is equivalent to that $\mu\in \mathfrak Q_\alpha^\infty(H)$ if and only if $\mu\in I(\rd)$ and for each $b>1$ there exists 
$\rho_b\in \mathfrak Q_\alpha^\infty(H)$ satisfying
\eqref{def_characteristic_function}.
\item Let $m\in\{0,1,2,\dots,\infty\}$. If $\mathfrak Q_\alpha(H)\subset H$ for all $\alpha\in(0,2]$, 
then $\mathfrak Q_\alpha^m(H)$ is decreasing in $\alpha\in\R$ with respect to set inclusion, namely,
\begin{equation}\label{decreasing in alpha}
\mathfrak Q_{\alpha_1}^{m}(H)\supset \mathfrak Q_{\alpha_2}^{m}(H)\quad\text{for }\alpha_1<\alpha_2.
\end{equation}
\end{enumerate}
\end{prop}
\begin{proof}
(i) Let us show the statement for $m\in\Z_+$ by induction. The case for $m=0$ is obvious.
Assume that $\mathfrak Q_\alpha^{m-1}(H)$ is c.c.s.s. 
Then, Theorem \ref{characterization_c.f.} (i) yields that
$\mu\in\mathfrak Q_\alpha^m(H)$ if and only if 
$\mu\in I(\rd)$ and
for each $b>1$ there exists $\rho_b\in\mathfrak Q_\alpha^{m-1}(H)$ satisfying
\eqref{def_characteristic_function}.
By using this decomposability, it is easy to see that $\mathfrak Q_\alpha^m(H)$ is c.c.s.s.
Thus $\mathfrak Q_\alpha^m(H)$ is c.c.s.s. for all $m\in\Z_+$.
Recalling that the intersection of c.c.s.s.\,classes is again c.c.s.s., we have the assertion for $m=\infty$.

(ii) Noting (i), we can apply Theorem \ref{characterization_c.f.} to the class $\mathfrak Q_\alpha^m(H)$ in place of $H$.

(iii) We first show the case for $\alpha\in(-\infty,0]\cup(2,\infty)$. It follows from Corollary \ref{Q_alpha(H) H} that 
$\mathfrak Q_\alpha(H)\subset H$. Then Lemma \ref{Inclusion} yields \eqref{decreasing in m}.
We next show the case for $\alpha\in(0,2)$. Suppose that $m\in\N$ and $\mu\in \mathfrak Q_\alpha^{m+1}(H)$. 
Then it follows from (ii) that
for each $b>1$ there exists $\rho_b\in \mathfrak Q_\alpha^m(H)$ satisfying
\eqref{def_characteristic_function}.
Then $\mu\in L^{\la\alpha\ra}(\rd)$ and hence $\rho_b\in I_\alpha(\rd)\subset\mathcal C_\alpha(\rd)$
by Lemma \ref{rho_b in I_alpha}.
Therefore $\rho_b\in \mathfrak Q_\alpha^m(H)\cap \mathcal C_\alpha(\rd)\subset\mathfrak Q_\alpha^{m-1}(H)$ 
by Corollary \ref{Q_alpha(H) H}.
Then it follows from (ii) that $\mu\in\mathfrak Q_\alpha^m(H)$. Thus \eqref{decreasing in m} holds.
We finally show the case for $\alpha=2$.
If $H=\emptyset$, then $\mathfrak Q_\alpha^m(H)=\emptyset$ for $m\in\N$ and thus \eqref{decreasing in m} is true. 
Let $H\neq\emptyset$.
It is sufficient to show that 
\begin{equation}\label{Gauss}
\text{for all $m\in\N$, $\mathfrak Q_2^m(H)$ is the class of all Gaussian distributions.}
\end{equation}
Let us show this statement by induction. If $m=1$, the assertion is Corollary \ref{alpha>=2}.
Assume that the assertion is valid for $m$.
Then $\mathfrak Q_2^{m+1}(H)=\mathfrak Q_2\left(\mathfrak Q_2^m(H)\right)=\mathfrak Q_2
\left(\{\mu\in \mathscr P(\rd) \colon\text{\(\mu\) is Gaussian}\}\right)$, which is equal to the class of all Gaussian distributions 
on $\rd$ by Corollary \ref{alpha>=2}.
Then the statement is true for $m+1$. 
Therefore the statement is true for all $m\in\N$.

(iv) It follows from (iii) that $\mathfrak Q_\alpha^{m}(H)\supset \mathfrak Q_\alpha^\infty(H)$ for all $m\in\N$. 
Then, Lemma \ref{Inclusion} entails that $\mathfrak Q_\alpha^{m+1}(H)\supset \mathfrak Q_\alpha
\left(\mathfrak Q_\alpha^\infty(H)\right)$ 
for all $m\in\N$. 
Therefore $\mathfrak Q_\alpha^\infty(H)\supset \mathfrak Q_\alpha\left(\mathfrak Q_\alpha^\infty(H)\right)$.
To prove the converse inclusion, let $\mu\in \mathfrak Q_\alpha^\infty(H)$. 
Then $\mu\in \mathfrak Q_\alpha^{m+1}(H)$ for all $m\in\Z_+$.
Therefore it follows from (ii) that for any $b>1$ there exists $\rho_{m,b}\in \mathfrak Q_\alpha^m(H)$ such that 
$\widehat{\mu}(z)=\widehat{\mu}(b^{-1}z)^{b^\alpha}\widehat{\rho}_{m,b}(z)$. 
Since $\mu\in I(\rd)$, $\widehat{\mu}(b^{-1}z)^{b^\alpha}$ does not vanish. 
Therefore $\widehat{\rho}_{m,b}(z)=\widehat{\mu}(z)/\widehat{\mu}(b^{-1}z)^{b^\alpha}$, which is independent of $m$. 
Denoting it by $\widehat{\rho}_{\infty,b}(z)$, we have $\rho_{\infty,b}\in \mathfrak Q_\alpha^m(H)$ for all $m\in\Z_+$, 
namely, $\rho_{\infty,b}\in\bigcap_{m=0}^\infty\mathfrak Q_\alpha^m(H)\subset\mathfrak Q_\alpha^\infty(H)$. 
Then $\mu\in \mathfrak Q_\alpha\left(\mathfrak Q_\alpha^\infty(H)\right)$ by Theorem \ref{characterization_c.f.}. 
Hence $\mathfrak Q_\alpha^\infty(H)\subset \mathfrak Q_\alpha\left(\mathfrak Q_\alpha^\infty(H)\right)$.

(v) Note that $\mathfrak Q_\alpha(H)\subset H$ for all $\alpha\in\R$
by Corollary \ref{Q_alpha(H) H} (i) and the assumption.
Let us show the statement for $m\in\Z_+$ by induction.
The case for $m=0$ is trivial.
Assume that the assertion is valid for $m-1$.
If $\mu\in \mathfrak Q_{\alpha_2}^m(H)$, then, by (ii), for each $b>1$ there exists $\rho_b\in \mathfrak Q_{\alpha_2}^{m-1}(H)$ 
satisfying \eqref{def_characteristic_function} 
for $\alpha_2$ in place of $\alpha$.
Noting that $b^{\alpha_2}-b^{\alpha_1}>0$, we have
$$
\widehat{\mu}(z)=\widehat{\mu}(b^{-1}z)^{b^{\alpha_1}}\left\{\widehat{\mu}(b^{-1}z)^{b^{\alpha_2}-b^{\alpha_1}}
\widehat\rho_b(z)\right\}.
$$
By (iii), we have $\mu\in \mathfrak Q_{\alpha_2}^m(H)\subset \mathfrak Q_{\alpha_2}^{m-1}(H)$.
Then, the assumption of induction entails that $\mu,\rho_b\in\mathfrak Q_{\alpha_2}^{m-1}(H)\subset 
\mathfrak Q_{\alpha_1}^{m-1}(H)$.
Since $\mathfrak Q_{\alpha_1}^{m-1}(H)$ is c.c.s.s. from (i),
the distribution with characteristic function $\widehat{\mu}(b^{-1}z)^{b^{\alpha_2}-b^{\alpha_1}}\widehat\rho_b(z)$ 
also belongs to $\mathfrak Q_{\alpha_1}^{m-1}(H)$. Hence $\mu\in \mathfrak Q_{\alpha_1}(H)$ by virtue of (ii).
Therefore the statement is true for all $m\in\Z_+$.
Taking the intersection under $m\in\N$ of the both sides of \eqref{decreasing in alpha}, we have 
the assertion for $m=\infty$.
\end{proof}
For $H\subset \mathscr P(\rd)$, we write $\overline{H}$ for the closure of $H$ under weak convergence and convolution.
Some facts related to the class of stable distributions are the following.
\begin{prop}\label{supset Q_alpha}
Let $H\subset I(\rd)$ be c.c.s.s. and $m\in\{1,2,\dots,\infty\}$.
\begin{enumerate}[(i)]
\item If $\alpha\leq 0$ and $H\supset S(\rd)$, then $\mathfrak Q_\alpha^{m}(H)\supset \overline{S(\rd)}$.
\item If $0<\alpha<2$ and $H\supset \bigcup_{\beta\in[\alpha,2]}S_\beta(\rd)$, then $\mathfrak Q_\alpha^{m}(H)\supset 
\overline{\bigcup_{\beta\in[\alpha,2]}S_\beta(\rd)}$.
\item If $\alpha=2$ and $H\neq\emptyset$, then $\mathfrak Q_2^{m}(H)$ is the class of all Gaussian distributions.
\item If $\alpha>2$ and $H\neq\emptyset$, then $\mathfrak Q_\alpha^{m}(H)$ is the class of all $\delta$-distributions.
\end{enumerate}
\end{prop}
\begin{proof}
(i) Let $\mu\in S(\rd)$. 
Then, there exists $\beta\in(0,2]$ such that for each $b>1$ there is $c_b\in\rd$ satisfying 
$\wh\mu(z)=\wh\mu(b^{-1}z)^{b^\beta}e^{i\la c_b,z\ra}$. 
Noting that $\alpha<\beta$ and letting 
\begin{equation}\label{rho_b}
\wh\rho_b(z):=\wh\mu(b^{-1}z)^{b^\beta-b^\alpha}e^{i\la c_b,z\ra},
\end{equation}
we have \eqref{def_characteristic_function}. 
Since $\mu\in S(\rd)\subset H$ and hence $\rho_b\in H$,
it follows that $\mu\in \mathfrak Q_\alpha(H)$.
Then, looking at \eqref{rho_b} and taking into account that 
$\mathfrak Q_\alpha(H)$ is c.c.s.s., 
we have $\rho_b\in \mathfrak Q_\alpha(H)$, which implies $\mu\in \mathfrak Q_\alpha^2(H)$ 
by Proposition \ref{Q_alpha^m_property} (ii).
Iterating this argument, we have $\mu\in \mathfrak Q_\alpha^{m}(H)$ for all $m\in\N$.
Therefore $\mathfrak Q_\alpha^{m}(H)\supset S(\rd)$ for all $m\in\N$.
Since $\mathfrak Q_\alpha^{m}(H)$ is c.c.s.s., it follows that 
$\mathfrak Q_\alpha^{m}(H)\supset \overline{S(\rd)}$ for all $m\in\N$.
Thus $\mathfrak Q_\alpha^\infty(H)=\bigcap_{m=1}^\infty\mathfrak Q_\alpha^{m}(H)\supset \overline{S(\rd)}$.

(ii) It is proved in a similar way to (i).

(iii) For $m\in\N$, what we have to show is \eqref{Gauss} itself, which is already shown.
For $m=\infty$, we have that $\mathfrak Q_2^\infty(H)=\bigcap_{m=1}^\infty\mathfrak Q_2^{m}(H)=
\{\mu\in \mathscr P(\rd) \colon\text{\(\mu\) is Gaussian}\}$.

(iv) For $m\in\N$, the statement can be proved in the same way as that for \eqref{Gauss}.
For $m=\infty$, it is proved in the same way as (iii).
\end{proof}
We now define $L_m^{\la\al\ra}(\rd)$, the nested subclasses of $L^{\la\al\ra}(\rd)$.
Define
$L^{\la\alpha\ra}_m(\rd)$ by $\mathfrak Q_\alpha^{m+1}(I(\rd))$ for 
$\alpha\in\R$ and $m\in\{0,1,2,\dots,\infty\}$.
Take into account that $L^{\la\alpha\ra}_0(\rd)=L^{\la\alpha\ra}(\rd)$.
Noting that $\mathfrak Q_\alpha(I(\rd))\subset I(\rd)$ for all $\alpha\in(0,2]$ and 
$I(\rd)\supset S(\rd)$, we have the following two propositions immediately from Propositions 
\ref{Q_alpha^m_property} and \ref{supset Q_alpha}.
\begin{prop}\label{The properties of L^{<alpha>}_m}
The following hold.
\begin{enumerate}[(i)]
\item For $\alpha\in\R$ and $m\in\{0,1,2,\dots,\infty\}$, $L^{\la\alpha\ra}_m(\rd)$ is c.c.s.s.
\item For $\alpha\in\R$ and $m\in\Z_+$, $\mu\in L^{\la\alpha\ra}_{m+1}(\rd)$ if and only if $\mu\in I(\rd)$ and for each $b>1$
 there exists $\rho_b\in L^{\la\alpha\ra}_m(\rd)$ satisfying
\eqref{def_characteristic_function}.
\item Let $\alpha\in\R$. Then $L^{\la\alpha\ra}_m(\rd)\supset L^{\la\alpha\ra}_{m+1}(\rd)$ for $m\in \Z_+$.
\item Let $\alpha\in\R$. Then,
$\mathfrak Q_\alpha\left(L^{\la\alpha\ra}_\infty(\rd)\right)=L^{\la\alpha\ra}_\infty(\rd)$, namely,
$\mu\in L^{\la\alpha\ra}_\infty(\rd)$ if and only if $\mu\in I(\rd)$ and for each $b>1$ there exists 
$\rho_b\in L^{\la\alpha\ra}_\infty(\rd)$ satisfying
\eqref{def_characteristic_function}.
\item Let $m\in\{0,1,2,\dots,\infty\}$. Then 
$L^{\la\alpha_1\ra}_m(\rd)\supset L^{\la\alpha_2\ra}_m(\rd)$ for $\alpha_1<\alpha_2$.
\end{enumerate}
\end{prop}
\begin{prop}\label{supset}
Let $m\in\{0,1,2,\dots,\infty\}$.
\begin{enumerate}[(i)]
\item If $\alpha\leq 0$, then $L^{\la\alpha\ra}_m(\rd)\supset \overline{S(\rd)}$.
\item If $0<\alpha<2$, then $L^{\la\alpha\ra}_m(\rd)\supset \overline{\bigcup_{\beta\in[\alpha,2]}S_\beta(\rd)}$.
\item If $\alpha=2$, then $L^{\la 2\ra}_m(\rd)$ is the class of all Gaussian distributions.
\item If $\alpha>2$, then $L^{\la\alpha\ra}_m(\rd)$ is the class of all $\delta$-distributions.
\end{enumerate}
\end{prop}
We next characterize $L_m^{\la\alpha\ra}(\rd)$ in terms of L\'evy measures.  For
$m=0$, 
\linebreak
\citet{MaejimaUeda2009b} proved the following.
\begin{thm}\label{characterization radial}
Let $\alpha<2$. 
Then, $\mu\in I(\rd)$ with L\'evy measure $\nu$ belongs to $L^{\la\alpha\ra}_0(\rd)$ if and only if
$$
\nu(B)=\int_S\lambda(d\xi)\int_0^\infty\1_B(r\xi)r^{-\alpha-1}k_\xi(r)dr,\quad B\in\B(\rd\setminus \{0\}),
$$
where $\lambda$ is a probability measure on $S$ and $k_\xi(r)$ is right-continuous and nonincreasing in $r\in(0,\infty)$ 
and measurable in $\xi\in S$, and for all $\xi\in S$,
$$
\int_0^\infty (r^2\wedge 1)r^{-\alpha-1}k_\xi(r)dr=\int_\rd (|x|^2\wedge 1)\nu(dx),
$$
which is independent of $\xi$.
If $\nu\neq 0$, then this $\lambda$ is uniquely determined by $\nu$, and this $k_\xi(\cdot)$ is uniquely determined by 
$\nu$ up to $\xi$ of $\lambda$-measure $0$.
\end{thm}
For characterizations of $L_m^{\la\al\ra}(\rd)$, we need some preparation.
\begin{defn}
Let $\alpha<2$. 
For $\mu\in L^{\la\alpha\ra}_0(\rd)$ with L\'evy measure $\nu\neq 0$, we call $k_\xi(r)$ in 
Theorem \ref{characterization radial} the \emph{$k$-function} of $\nu$ (or $\mu$).
If $\nu=0$, then we define the $k$-function of $\nu$ (or $\mu$) as the zero-function.
And we call the function $h_\xi(u),u\in\R$ defined by $h_\xi(u):=k_\xi(e^{-u})$ the \emph{$h$-function} of $\nu$ (or $\mu$).
\end{defn}
For $f\colon\R\to\R$, we introduce the \emph{difference operator} as follows:
$$
\Delta_\varepsilon^n f(u):=\sum_{j=0}^n(-1)^{n-j}\binom{n}{j}f(u+j\varepsilon),\quad \text{for }u\in\R,\ \varepsilon >0
\text{ and }n\in\Z_+.
$$
For $m\in\Z_+$, $f\colon\R\to\R$ is said to be \emph{monotone of order $m$} if $\Delta_\varepsilon^n f(u)\geq 0$ 
for all $u\in\R$, $\varepsilon >0$ and $n=0,1,2,\dots,m$. $f\colon\R\to\R$ is said to be \emph{absolutely monotone} 
if $f$ is monotone of order $m$ for all $m\in\Z_+$.

The following four statements are proved by similar arguments to those in Section 1.2 of \citet{Sato's_book2003}, originally done in \citet{Sato1980},
so we omit their proofs.
\begin{thm}\label{monotone h-function}
Suppose $\alpha<2$.
\begin{enumerate}[(i)]
\item Let $m\in\Z_+$. Then $\mu\in L^{\la\alpha\ra}_m(\rd)$ if and only if $\mu\in L^{\la\alpha\ra}_0(\rd)$ and 
the $h$-function $h_\xi(u)$ of $\mu$ is monotone of order $m+1$ in $u\in\R$ for $\lambda$-a.e. $\xi\in S$.
\item $\mu\in L^{\la\alpha\ra}_\infty(\rd)$ if and only if $\mu\in L^{\la\alpha\ra}_0(\rd)$ and the $h$-function 
$h_\xi(u)$ of $\mu$ is absolutely monotone in $u\in\R$ for $\lambda$-a.e. $\xi\in S$.
\end{enumerate}
\end{thm}
\begin{lem}\label{Gamma_xi}
Let $\alpha<2$ and $0<c<\infty$.
A function $h_\xi(u)$ is absolutely monotone in $u\in\R$ and measurable in $\xi\in S$ and satisfies
$$
\int_{-\infty}^\infty\left(e^{-2u}\wedge 1\right)e^{\alpha u}h_\xi(u)du=c
$$
for all $\xi\in S$
if and only if
$$
e^{\alpha u}h_\xi(u)=\int_{(0,2)\cap[\alpha,2)}e^{\beta u}\Gamma_\xi(d\beta),
$$
where $\Gamma_\xi$ is a measure on $(0,2)\cap[\alpha,2)$ for each $\xi\in S$ satisfying
$$
\int_{(0,2)\cap[\alpha,2)}\left(\frac 1\beta+\frac 1{2-\beta}\right)\Gamma_\xi(d\beta)=c
$$
and $\Gamma_\xi(B)$ is measurable in $\xi\in S$ for every $B\in\B\left((0,2)\cap[\alpha,2)\right)$.
\end{lem}
\begin{thm}\label{L^<alpha>_infty}
Let $\alpha<2$.
\begin{enumerate}[(i)]
\item If $\mu\in L^{\la\alpha\ra}_\infty(\rd)$ with L\'evy measure $\nu$, then
$$
\nu(B)=\int_{(0,2)\cap[\alpha,2)}\Gamma(d\beta)\int_S\lambda_\beta(d\xi)\int_0^\infty\1_B(r\xi)r^{-\beta-1}dr,
\quad B\in\B(\rd\setminus \{0\}),
$$
where $\Gamma$ is a measure on $(0,2)\cap[\alpha,2)$ satisfying
$$
\int_{(0,2)\cap[\alpha,2)}\left(\frac 1\beta+\frac 1{2-\beta}\right)\Gamma(d\beta)<\infty,
$$
and $\lambda_\beta$ is a probability measure on $S$ for each $\beta\in (0,2)\cap[\alpha,2)$, and $\lambda_\beta(C)$ is 
measurable in $\beta\in (0,2)\cap[\alpha,2)$ for every $C\in\B(S)$.
This $\Gamma$ is uniquely determined by $\mu$ and this $\lambda_\beta$ is uniquely determined by $\mu$ up to 
$\beta$ of $\Gamma$-measure $0$.
\item If $\mu\in I(\rd)$ with L\'evy measure $\nu$ is expressible as in (i), then $\mu\in L^{\la\alpha\ra}_\infty(\rd)$.
\end{enumerate}
\end{thm}
\begin{thm}
\begin{enumerate}[(i)]
\item If $\alpha\leq 0$, then
$L^{\la\alpha\ra}_\infty(\rd)\subset\overline{S(\rd)}$.
\item If $0<\alpha<2$, then 
$L^{\la\alpha\ra}_\infty(\rd)\subset\overline{\bigcup_{\beta\in[\alpha,2]}S_\beta(\rd)}$.
\end{enumerate}
\end{thm}
Combining this theorem with Proposition \ref{supset} with $m=\infty$, we conclude 
\begin{thm}\label{L infty}
\begin{enumerate}[(i)]
\item If $\alpha\leq 0$, then
$L^{\la\alpha\ra}_\infty(\rd)=\overline{S(\rd)}$.
\item If $0<\alpha<2$, then 
$L^{\la\alpha\ra}_\infty(\rd)=\overline{\bigcup_{\beta\in[\alpha,2]}S_\beta(\rd)}$.
\end{enumerate}
\end{thm}
To conclude this section, we go back once to the case for a general c.c.s.s.\ $H\subset I(\rd)$.
\begin{thm}
Let $H\subset I(\rd)$ be c.c.s.s.
\begin{enumerate}[(i)]
\item If $\alpha\leq 0$ and $H\supset S(\rd)$, then $\mathfrak Q_\alpha^\infty(H)=\overline{S(\rd)}$.
\item If $0<\alpha<2$ and $H\supset \bigcup_{\beta\in[\alpha,2]}S_\beta(\rd)$, 
then $\mathfrak Q_\alpha^\infty(H)= \overline{\bigcup_{\beta\in[\alpha,2]}S_\beta(\rd)}$.
\end{enumerate}
\end{thm}
\begin{proof}
We only prove (i), since (ii) is similarly proved.
Proposition \ref{supset Q_alpha} yields that
$\mathfrak Q_\alpha^\infty(H)\supset \overline{S(\rd)}$.
Using Lemma \ref{Inclusion} repeatedly, we have $\mathfrak Q_\alpha^{m+1}(H)\subset \mathfrak Q_\alpha^{m+1}(I(\rd))=
L^{\la\alpha\ra}_m(\rd)$ for $m\in\Z_+$. Hence $\mathfrak Q_\alpha^\infty(H)\subset L^{\la\alpha\ra}_\infty(\rd)=\overline{S(\rd)}$ by Theorem \ref{L infty}.
Thus we have $\mathfrak Q_\alpha^\infty(H)=\overline{S(\rd)}$.
\end{proof}
\vskip 10mm
\section{Nested subclasses of the class of $\alpha$-selfdecomposable distributions in terms of mapping}
For $\alpha\in\R$,
\citet{MaejimaMatsuiSuzuki} defined mappings
$\Phi _{\alpha}\colon {\mathfrak D}(\Phi_{\alpha})\rightarrow I(\rd)$
by
\begin{equation}\label{Phialpha}
\Phi _{\alpha} (\mu) =
\begin{cases}
\displaystyle\law\left (\int _0^{-1/\alpha} (1+\alpha t)^{-1/\alpha}dX_t^{(\mu)}\right ),&\text{when }\alpha<0,\\[15pt]
\displaystyle\law\left (\int _0^{\infty} e^{-t}dX_t^{(\mu)}\right ),&\text{when }\alpha=0,\\[15pt]
\displaystyle\law\left (\int _0^{\infty} (1+\alpha t)^{-1/\alpha}dX_t^{(\mu)}\right ),&\text{when }\alpha>0.\\
\end{cases}
\end{equation}
Due to Theorems 2.4 and 2.8 of \citet{Sato2006},
the domains $\mathfrak D(\Phi_{\alpha})$ are as follows, (see also p.\ 49 of \citet{Sato2006}).
$$
\mathfrak  D (\Phi_{\alpha})=
\begin{cases}
I(\rd),&\text{when }\alpha <0,\\
I_{\log}(\rd),&\text{when }\alpha=0,\\
I_{\alpha}(\rd),&\text{when }0<\alpha <1,\\
I_1^*(\rd),&\text{when }\alpha=1,\\
I_\alpha^0(\rd),&\text{when }1<\alpha <2,\\
\{\delta_0\},&\text{when }\alpha\geq 2,
\end{cases}
$$
where
\begin{align*}
I_\alpha^0(\rd)
&=\left\{\mu\in 
I_\alpha(\rd)\colon \int_{\rd}x\mu(dx)=0\right\},\quad\text{for }1\leq \alpha<2,\\
I_1^*(\rd)
&=\left\{\mu=\mu_{(A,\nu,\gamma)}\in I_1^0(\rd)\colon\lim_{T\to\infty}\int_{1}^Tt^{-1}dt\int_{|x|>t}x\nu(dx)\text{ exists in }\rd\right\}.
\end{align*}
As to the ranges $\mathfrak R (\Phi_{\alpha})$, Theorem 4.6 of \citet{MaejimaMatsuiSuzuki} says the following.
\begin{equation}\label{Phi range}
\mathfrak R (\Phi_{\alpha})=
\begin{cases}
L^{\la\alpha\ra}(\rd),&\text{when }\alpha <0,\\
L^{\la0\ra}(\rd),&\text{when }\alpha=0,\\
L^{\la\alpha\ra}(\rd)\cap \mathcal C_\alpha(\rd),&\text{when }0<\alpha <1,\\
L^{\la1\ra}(\rd)\cap \mathcal C_1^*(\rd),&\text{when }\alpha=1,\\
L^{\la\alpha\ra}(\rd)\cap \mathcal C_\alpha^0(\rd),&\text{when }1<\alpha <2,\\
\{\delta_0\},&\text{when }\alpha\geq 2,
\end{cases}
\end{equation}
where
\begin{align*}
\mathcal C_1^*(\rd)
&=\left\{\wt\mu_{(\wt A,\wt\nu,\wt\gamma)}\in 
L^{\la1\ra}(\rd)\cap\mathcal C_1(\rd)\colon \wt\nu(B)=\int_S\wt\lambda(d\xi)\int_0^\infty \1_B(r\xi)r^{-2}\wt k_\xi(r)dr,\right.\\ 
&\hspace{45pt}\left.\lim_{\varepsilon \downarrow 0}
\int_\varepsilon^1t dt \int_S\xi\wt\lambda(d\xi)\int_0^\infty 
\frac{r^2}{1+t^2r^2}d\wt k_\xi(r)\ \text{exists in \(\rd\) and
 equals \(\widetilde\gamma\)}\right\},\\
\mathcal C_\alpha^0(\rd)
&=\mathcal C_\alpha(\rd)\cap I_1^0(\rd),\quad\text{for }1< \alpha<2.
\end{align*}

Now, we characterize $\Phi_{\alpha}^m\left(H\cap \mathfrak D(\Phi_{\alpha}^m)\right)$ with a c.c.s.s.\,$H\subset I(\rd)$ 
using the results in the previous section.
Note that, for $\alpha<0$, $\mathfrak D(\Phi_{\alpha}^m)=I(\rd),m\in\N$, since $\mathfrak D(\Phi_{\alpha})=I(\rd)$.
However, henceforth we do not treat the case for $\alpha\geq 2$, since it is obvious 
that $\Phi_\alpha^m(\{\delta_0\})=\{\delta_0\}$ for all $m\in\N$.
\begin{thm}\label{Phi_alpha^m(H)}
Let $H\subset I(\rd)$ be c.c.s.s., and let $m\in\N$.
\begin{enumerate}[(i)]
\item When $\alpha < 0$,  $\Phi_{\alpha}^m\left(H\right) = \mathfrak Q_\alpha^m(H)$.
\item When $\alpha = 0$,  $\Phi_0^m\left(H\cap \mathfrak D(\Phi_0^m)\right) = \mathfrak Q_0^m(H)$.
\item When $ 0<\alpha<1$, $\Phi_{\alpha}^m\left(H\cap \mathfrak D(\Phi_{\alpha}^m)\right) = 
\mathfrak Q_\alpha^m(H)\cap \mathcal C_\alpha(\rd)$.
\item When $\alpha =1$,  $\Phi_1^m\left(H\cap \mathfrak D(\Phi_1^m)\right) = 
\mathfrak Q_1^m(H)\cap \mathcal C_1^*(\rd)$.
\item When $1<\alpha<2$, $\Phi_{\alpha}^m\left(H\cap \mathfrak D(\Phi_{\alpha}^m)\right) = 
\mathfrak Q_\alpha^m(H)\cap \mathcal C_\alpha^0(\rd)$.
\end{enumerate}
\end{thm}
\begin{proof}
(i) It is proved in a similar way to (v).

(ii) We prove the statement by induction. The case for $m=0$ comes from Lemma 4.1 of \citet{Barndorff-NielsenMaejimaSato2006} and Theorem \ref{characterization_c.f.} (i) with $\alpha=0$ of this paper.
Now assume that the statement is valid for $m-1$ with $m\geq 2$ in place of $m$ and let us prove
$\Phi_0^m\left(H\cap \mathfrak D(\Phi_0^m)\right) = \mathfrak Q_0^m(H)$. 
If we put $H':=\Phi_0\left(H\cap \mathfrak D(\Phi_0)\right)$, then it is equal to $\mathfrak Q_0(H)$
by the statement for $m=0$ and thus it is c.c.s.s. Applying the assumption of induction to $H'$ instead of $H$, we have that $\Phi_0^{m-1}\left(\Phi_0\left(H\cap \mathfrak D(\Phi_0)\right)\cap \mathfrak D(\Phi_0^{m-1})\right) = \mathfrak Q_0^m(H)$. Since it is easy to see that $\Phi_0\left(H\cap \mathfrak D(\Phi_0)\right)\cap \mathfrak D(\Phi_0^{m-1})=\Phi_0\left(H\cap \mathfrak D(\Phi_0^m)\right)$, it follows that $\Phi_0^m\left(H\cap \mathfrak D(\Phi_0^m)\right) = \mathfrak Q_0^m(H)$.

(iii) It is proved in a similar way to (v).

(iv) We prove the statement by induction. 
Let us prove the case for $m=1$. We first show that 
$\Phi_1\left(H\cap \mathfrak D(\Phi_1)\right) \subset \mathfrak Q_1(H)\cap \mathcal C_1^*(\rd)$.
If $\mu\in\Phi_1\left(H\cap \mathfrak D(\Phi_1)\right)$, 
then $\mu=\Phi_1(\mu_0)$ for some $\mu_0\in H\cap \mathfrak D(\Phi_1)$.
We have, for any $b>1$ and $z\in\rd$,
$$
C_\mu(z)-b C_\mu(b^{-1}z)
=\int_0^{b-1}C_{\mu_0}\left((1+t)^{-1}z\right)dt
=C_{\rho_b}(z),
$$
where 
$$
\rho_b=\law\left(\int_0^{b-1}(1+t)^{-1}dX_t^{(\mu_0)}\right).
$$
Since $H$ is c.c.s.s., $\rho_b\in H$ for all $b>1$.
Then it follows from Theorem \ref{characterization_c.f.} (i) that $\mu\in\mathfrak Q_1(H)$.
Since $\mu\in\mathfrak R(\Phi_1)\subset \mathcal C_1^*(\rd)$, 
we have $\mu\in\mathfrak Q_1(H)\cap \mathcal C_1^*(\rd)$.
We next show that 
$\Phi_1\left(H\cap \mathfrak D(\Phi_1)\right) \supset \mathfrak Q_1(H)\cap \mathcal C_1^*(\rd)$.
If $\mu\in \mathfrak Q_1(H)\cap \mathcal C_1^*(\rd)$, then $\mu\in L^{\la1\ra}(\rd)\cap \mathcal C_1^*(\rd)$ and hence
$\mu=\Phi_1(\mu_0)$ for some $\mu_0\in\mathfrak D(\Phi_1)$.
On the other hand, due to Theorem \ref{characterization_c.f.} (i), for each $b>1$, there is 
$\rho_b\in H$ satisfying \eqref{def_characteristic_function} with $\alpha=1$.
Then, it follows that
\begin{align}\label{C mu_0 alpha=1}
C_{\mu_0}(z)
&=\lim_{b\downarrow 1}\frac{1}{b-1}\int_0^{b-1}C_{\mu_0}\left((1+ t)^{-1}z\right)dt\\
\notag&=\lim_{b\downarrow 1}\frac{1}{b-1}\left\{C_\mu(z)-b C_\mu(b^{-1}z)\right\}
=\lim_{b\downarrow 1}\frac{1}{b-1}C_{\rho_b}(z).
\end{align}
This entails $\mu_0\in H$ since $H$ is c.c.s.s.
Then $\mu=\Phi_1(\mu_0)\in \Phi_1\left(H\cap \mathfrak D(\Phi_1)\right)$.
Therefore the case for $m=0$ is proved.
Now assume that the statement is valid for $m-1$ with $m\geq 2$ in place of $m$ and let us prove
$\Phi_1^m\left(H\cap \mathfrak D(\Phi_1^m)\right) = \mathfrak Q_1^m(H)\cap \mathcal C_1^*(\rd)$. 
We first show that $\Phi_1^m\left(H\cap \mathfrak D(\Phi_1^m)\right) \subset  \mathfrak Q_1^m(H)\cap \mathcal C_1^*(\rd)$. 
If $\mu\in \Phi_{\alpha}^m\left(H\cap \mathfrak D(\Phi_1^m)\right)$, then $\mu=\Phi_1^m(\mu_0)$ for some 
$\mu_0\in H\cap\mathfrak D(\Phi_1^m)$.
We have, for any $b>1$ and $z\in\rd$,
$$
C_\mu(z)-b C_\mu(b^{-1}z)
=\int_0^{b-1}C_{\Phi_1^{m-1}(\mu_0)}\left((1+ t)^{-1}z\right)dt
=C_{\rho_b}(z),
$$
where 
$$
\rho_b=\law\left(\int_0^{b-1}(1+ t)^{-1}dX_t^{\left(\Phi_1^{m-1}(\mu_0)\right)}\right).
$$
Since $\Phi_1^{m-1}(\mu_0)\in \mathfrak Q_1^{m-1}(H)\cap \mathcal C_1^*(\rd)$ by the assumption of induction and 
$\mathfrak Q_1^{m-1}(H)$ is c.c.s.s., we have
$\rho_b\in\mathfrak  Q_1^{m-1}(H)$ for each $b>1$.
Then, $\mu\in\mathfrak  Q_1^m(H)$ due to Proposition \ref{Q_alpha^m_property} (ii).
Since $\mu\in \mathfrak R(\Phi_1^m)\subset \mathfrak R(\Phi_1)\subset \mathcal C_1^*(\rd)$, we have
$\mu\in\mathfrak Q_1^m(H)\cap \mathcal C_1^*(\rd)$.
We next show that $\Phi_1^m\left(H\cap \mathfrak D(\Phi_1^m)\right) \supset \mathfrak Q_1^m(H)\cap \mathcal C_1^*(\rd)$.
If $\mu\in \mathfrak Q_1^m(H)\cap \mathcal C_1^*(\rd)$, then $\mu\in L^{\la1\ra}(\rd)\cap \mathcal C_1^*(\rd)$
since $\mathfrak Q_1^m(H)\subset \mathfrak Q_1^m(I(\rd))=L^{\la1\ra}_{m-1}(\rd)\subset L^{\la1\ra}(\rd)$.
Hence
$\mu=\Phi_1(\mu_0)$ for some $\mu_0\in\mathfrak D(\Phi_1)$.
On the other hand, due to Proposition \ref{Q_alpha^m_property} (ii), for each $b>1$, there is 
$\rho_b\in\mathfrak Q_1^{m-1}(H)$ satisfying \eqref{def_characteristic_function}.
Then, \eqref{C mu_0 alpha=1} holds.
Since $\mathfrak Q_1^{m-1}(H)$ is c.c.s.s.,
$\mu_0\in \mathfrak Q_1^{m-1}(H)$.
Noting that $\mu_0\in\mathfrak D(\Phi_1)=I_1^*(\rd)$, we have $\mu_0\in \mathfrak Q_1^{m-1}(H)\cap I_1^*(\rd)$.
Since $\mathfrak Q_1^{m-1}(H)\subset\mathfrak Q_1^{m-1}(I(\rd))=L^{\la1\ra}_{m-2}(\rd)\subset L^{\la1\ra}(\rd)$, we have that 
$\mu_0={\mu_0}_{(A_0,\nu_0,\gamma_0)}\in L^{\la1\ra}(\rd)\cap \mathcal C_1(\rd)$.
Therefore $\nu_0$ has the polar decomposition as follows:
$$
\nu_0(B)=\int_S\lambda_0(d\xi)\int_0^\infty\1_B(r\xi)r^{-2}k_{0,\xi}(r)dr,\quad B\in\B(\rd\setminus \{0\}),
$$
where $k_{0,\xi}(r)$ is right-continuous and nonincreasing in $r\in(0,\infty)$ 
and measurable in $\xi\in S$, and satisfies $\lim_{r\to\infty}k_{0,\xi}(r)=0$ for each $\xi\in S$.
Then Lemma 5.1 and its proof of \citet{MaejimaMatsuiSuzuki} yield that
$$
\nu_0(B)=\int_0^1\nu_1(s^{-1}B)s^{-2}ds,\ B\in \B(\rd\setminus \{0\}),\qquad \int_\rd(|x|^2\wedge |x|)\nu_1(dx)<\infty,
$$
with
$$
\nu_1(B)=-\int_S\lambda_0(d\xi)\int_0^\infty\1_B(r\xi)r^{-1}dk_{0,\xi}(r),\quad B\in\B(\rd\setminus \{0\}).
$$
Taking into account that $\mu_0\in I_1^*(\rd)$, we have $\int_\rd x\mu_0(dx)=0$, which is equivalent to that
\begin{align*}
\gamma_0&
=-\int_\rd\frac{x|x|^2}{1+|x|^2}\nu_0(dx)
=-\int_0^1s^{-2}ds\int_\rd\frac{sx|sx|^2}{1+|sx|^2}\nu_1(dx)\\
&=\int_0^1s ds \int_S\xi\lambda_0(d\xi)\int_0^\infty 
\frac{r^2}{1+s^2r^2}dk_{0,\xi}(r).
\end{align*}
This yields $\mu_0\in \mathcal C_1^*(\rd)$ and hence $\mu_0\in \mathfrak Q_1^{m-1}(H)\cap \mathcal C_1^*(\rd)$.
By the assumption of induction, we have $\mu_0=\Phi_1^{m-1}(\mu_2)$ for some $\mu_2\in H\cap\mathfrak D(\Phi_1^{m-1})$.
Then $\mu=\Phi_1(\mu_0)=\Phi_1^m(\mu_2)\in\Phi_1^m\left(H\cap \mathfrak D(\Phi_1^m)\right)$.

(v) We prove the statement by induction. 
Let us prove the case for $m=1$. We first show that 
$\Phi_\alpha\left(H\cap \mathfrak D(\Phi_\alpha)\right) \subset \mathfrak Q_\alpha(H)\cap \mathcal C_\alpha^0(\rd)$.
If $\mu\in\Phi_\alpha\left(H\cap \mathfrak D(\Phi_\alpha)\right)$, 
then $\mu=\Phi_\alpha(\mu_0)$ for some $\mu_0\in H\cap \mathfrak D(\Phi_\alpha)$.
We have, for any $b>1$ and $z\in\rd$,
$$
C_\mu(z)-b^\alpha C_\mu(b^{-1}z)
=\int_0^{(b^\alpha-1)/\alpha}C_{\mu_0}\left((1+\alpha t)^{-1/\alpha}z\right)dt
=C_{\rho_b}(z),
$$
where 
$$
\rho_b=\law\left(\int_0^{(b^\alpha-1)/\alpha}(1+\alpha t)^{-1/\alpha}dX_t^{(\mu_0)}\right).
$$
Since $H$ is c.c.s.s., $\rho_b\in H$ for all $b>1$.
Then it follows from Theorem \ref{characterization_c.f.} (i) that $\mu\in\mathfrak Q_\alpha(H)$.
Since $\mu\in\mathfrak R(\Phi_\alpha)\subset \mathcal C_\alpha^0(\rd)$, 
we have $\mu\in\mathfrak Q_\alpha(H)\cap \mathcal C_\alpha^0(\rd)$.
We next show that 
$\Phi_\alpha\left(H\cap \mathfrak D(\Phi_\alpha)\right) \supset \mathfrak Q_\alpha(H)\cap \mathcal C_\alpha^0(\rd)$.
If $\mu\in \mathfrak Q_\alpha(H)\cap \mathcal C_\alpha^0(\rd)$, 
then $\mu\in L^{\la\alpha\ra}(\rd)\cap \mathcal C_\alpha^0(\rd)$ and hence
$\mu=\Phi_\alpha(\mu_0)$ for some $\mu_0\in\mathfrak D(\Phi_\alpha)$.
On the other hand, due to Theorem \ref{characterization_c.f.} (i), for each $b>1$, there is 
$\rho_b\in H$ satisfying \eqref{def_characteristic_function}.
Then, it follows that
\begin{align}\label{C mu_0}
C_{\mu_0}(z)
&=\lim_{b\downarrow 1}\frac{\alpha}{b^\alpha-1}\int_0^{(b^\alpha-1)/\alpha}C_{\mu_0}\left((1+\alpha t)^{-1/\alpha}z\right)dt\\
\notag&=\lim_{b\downarrow 1}\frac{\alpha}{b^\alpha-1}\left\{C_\mu(z)-b^\alpha C_\mu(b^{-1}z)\right\}
=\lim_{b\downarrow 1}\frac{\alpha}{b^\alpha-1}C_{\rho_b}(z).
\end{align} 
This entails $\mu_0\in H$ since $H$ is c.c.s.s.
Then $\mu=\Phi_\alpha(\mu_0)\in \Phi_\alpha\left(H\cap \mathfrak D(\Phi_\alpha)\right)$.
Therefore the case for $m=0$ is proved.
Now assume that the statement is valid for $m-1$ with $m\geq 2$ in place of $m$ and let us prove
$\Phi_{\alpha}^m\left(H\cap \mathfrak D(\Phi_{\alpha}^m)\right) = \mathfrak Q_\alpha^m(H)\cap \mathcal C_\alpha^0(\rd)$. 
We first show that $\Phi_{\alpha}^m\left(H\cap \mathfrak D(\Phi_{\alpha}^m)\right) 
\subset  \mathfrak Q_\alpha^m(H)\cap \mathcal C_\alpha^0(\rd)$. 
If $\mu\in \Phi_{\alpha}^m\left(H\cap \mathfrak D(\Phi_{\alpha}^m)\right)$, then $\mu=\Phi_\alpha^m(\mu_0)$ for some 
$\mu_0\in H\cap\mathfrak D(\Phi_\alpha^m)$.
We have, for any $b>1$ and $z\in\rd$,
$$
C_\mu(z)-b^\alpha C_\mu(b^{-1}z)
=\int_0^{(b^\alpha-1)/\alpha}C_{\Phi_\alpha^{m-1}(\mu_0)}\left((1+\alpha t)^{-1/\alpha}z\right)dt
=C_{\rho_b}(z),
$$
where 
$$\rho_b=\law\left(\int_0^{(b^\alpha-1)/\alpha}(1+\alpha t)^{-1/\alpha}dX_t^{\left(\Phi_\alpha^{m-1}(\mu_0)\right)}\right).
$$
Since $\Phi_\alpha^{m-1}(\mu_0)\in \mathfrak Q_\alpha^{m-1}(H)\cap \mathcal C_\alpha^0(\rd)$ 
by the assumption of induction and 
$\mathfrak Q_\alpha^{m-1}(H)$ is c.c.s.s., we have
$\rho_b\in\mathfrak  Q_\alpha^{m-1}(H)$ for each $b>1$.
Then, $\mu\in\mathfrak  Q_\alpha^m(H)$ due to Proposition \ref{Q_alpha^m_property} (ii).
Since $\mu\in \mathfrak R(\Phi_{\alpha}^m)\subset \mathfrak R(\Phi_{\alpha})\subset \mathcal C_\alpha^0(\rd)$, we have
$\mu\in\mathfrak  Q_\alpha^m(H)\cap \mathcal C_\alpha^0(\rd)$.
We next show that $\Phi_{\alpha}^m\left(H\cap \mathfrak D(\Phi_{\alpha}^m)\right) \supset  
\mathfrak Q_\alpha^m(H)\cap \mathcal C_\alpha^0(\rd)$.
If $\mu\in \mathfrak Q_\alpha^m(H)\cap \mathcal C_\alpha^0(\rd)$, 
then $\mu\in L^{\la\alpha\ra}(\rd)\cap \mathcal C_\alpha^0(\rd)$
since $\mathfrak Q_\alpha^m(H)\subset \mathfrak Q_\alpha^m(I(\rd))=L^{\la\alpha\ra}_{m-1}(\rd)\subset L^{\la\alpha\ra}(\rd)$.
Hence
$\mu=\Phi_\alpha(\mu_0)$ for some $\mu_0\in\mathfrak D(\Phi_\alpha)$.
On the other hand, due to Proposition \ref{Q_alpha^m_property} (ii), for each $b>1$, there is 
$\rho_b\in\mathfrak Q_\alpha^{m-1}(H)$ satisfying \eqref{def_characteristic_function}.
Then, \eqref{C mu_0} holds.
Since $\mathfrak Q_\alpha^{m-1}(H)$ is c.c.s.s.,
$\mu_0\in \mathfrak Q_\alpha^{m-1}(H)$.
Noting that $\mu_0\in\mathfrak D(\Phi_\alpha)=I_\alpha^0(\rd)\subset \mathcal C_\alpha^0(\rd)$, we have 
$\mu_0\in \mathfrak Q_\alpha^{m-1}(H)\cap \mathcal C_\alpha^0(\rd)$.
By the assumption of induction, we have $\mu_0=\Phi_\alpha^{m-1}(\mu_1)$ for some 
$\mu_1\in H\cap\mathfrak D(\Phi_\alpha^{m-1})$.
Then $\mu=\Phi_\alpha(\mu_0)=\Phi_\alpha^m(\mu_1)\in\Phi_{\alpha}^m\left(H\cap \mathfrak D(\Phi_{\alpha}^m)\right)$.
\end{proof}
Let $\alpha<2$ and let $H\subset I(\rd)$ be c.c.s.s.
Then it follows from Proposition \ref{Q_alpha^m_property} (iii) and the theorem above that
$\Phi_{\alpha}^m\left(H\cap \mathfrak D(\Phi_{\alpha}^m)\right),m\in\N$ are nested subclasses of 
$\Phi_{\alpha}\left(H\cap \mathfrak D(\Phi_{\alpha})\right)$.
Using the results in the previous section, we obtain 
the limit $\lim_{m\to\infty}\Phi_{\alpha}^m\left(H\cap 
\mathfrak D(\Phi_{\alpha}^m)\right)=\bigcap_{m=1}^\infty\Phi_{\alpha}^m\left(H\cap \mathfrak D(\Phi_{\alpha}^m)\right)$ as follows.
\begin{cor}\label{Phi_alpha^infty(H)}
Let $H\subset I(\rd)$ be c.c.s.s.
\begin{enumerate}[(i)]
\item If $\alpha< 0$ and $H\supset S(\rd)$, 
then
$$
\lim_{m\to\infty}\Phi_{\alpha}^m(H) = \mathfrak Q_\alpha^\infty(H)=\overline{S(\rd)}.
$$
\item If $\alpha= 0$ and $H\supset S(\rd)$, 
then
$$
\lim_{m\to\infty}\Phi_0^m(H\cap \mathfrak D(\Phi_0^m)) = \mathfrak Q_0^\infty(H)=\overline{S(\rd)}.
$$
\item If $0<\alpha<1$ and $H\supset \bigcup_{\beta\in[\alpha,2]}S_\beta(\rd)$, then 
\begin{align*}
\lim_{m\to\infty}\Phi_{\alpha}^m&\left(H\cap \mathfrak D(\Phi_{\alpha}^m)\right) = 
\mathfrak Q_\alpha^\infty(H)\cap \mathcal C_\alpha(\rd)\\
&=
\overline{\bigcup_{\beta\in[\alpha,2]}S_\beta(\rd)}\cap \mathcal C_\alpha(\rd)= 
L^{\la\alpha\ra}_\infty(\rd)\cap \mathcal C_\alpha(\rd).
\end{align*}
\item If $\alpha=1$ and $H\supset \bigcup_{\beta\in[1,2]}S_\beta(\rd)$, then 
\begin{align*}
\lim_{m\to\infty}\Phi_1^m&\left(H\cap \mathfrak D(\Phi_1^m)\right) = 
\mathfrak Q_1^\infty(H)\cap \mathcal C_1^*(\rd)\\
&=
\overline{\bigcup_{\beta\in[1,2]}S_\beta(\rd)}\cap \mathcal C_1^*(\rd)= L^{\la1\ra}_\infty(\rd)\cap \mathcal C_1^*(\rd).
\end{align*}
\item If $1<\alpha<2$ and $H\supset \bigcup_{\beta\in[\alpha,2]}S_\beta(\rd)$, then 
\begin{align*}
\lim_{m\to\infty}\Phi_{\alpha}^m&\left(H\cap \mathfrak D(\Phi_{\alpha}^m)\right) = 
\mathfrak Q_\alpha^\infty(H)\cap \mathcal C_\alpha^0(\rd)\\
&=
\overline{\bigcup_{\beta\in[\alpha,2]}S_\beta(\rd)}\cap \mathcal C_\alpha^0(\rd)= 
L^{\la\alpha\ra}_\infty(\rd)\cap \mathcal C_\alpha^0(\rd).
\end{align*}
\end{enumerate}
\end{cor}
\begin{rem}\label{application}
Let $\Phi_f$ be a stochastic integral mapping defined by \eqref{mapping}.
It is a interesting problem to characterize the limit $\lim_{m\to\infty}\mathfrak R(\Phi_f^{m+1})$
as in \citet{MaejimaSato2009}.
Corollary \ref{Phi_alpha^infty(H)} can be applied to this problem as follows.
Assume that $\Phi_f$ is decomposed in the form that $\Phi_f=\Phi_\alpha\circ\Phi_g=\Phi_g\circ\Phi_\alpha$ for some $\alpha\in(-\infty,2)$ and some stochastic integral mapping $\Phi_g$. 
Then $\mathfrak R(\Phi_f^m)=\Phi_\alpha^m\left(\mathfrak R(\Phi_g^m)\cap\mathfrak D(\Phi_\alpha^m)\right)$, so that
$\Phi_\alpha^m\left(H\cap\mathfrak D(\Phi_\alpha^m)\right)\subset \mathfrak R(\Phi_f^m)\subset\mathfrak R(\Phi_\alpha^m)$,
where $H=\lim_{m\to\infty}\mathfrak R(\Phi_g^{m+1})$.
If $H$ fulfills the conditions in Corollary \ref{Phi_alpha^infty(H)}, then we have $\lim_{m\to\infty}\mathfrak R(\Phi_f^{m+1})=\lim_{m\to\infty}\mathfrak R(\Phi_\alpha^{m+1})$.
An example of this application is found in \citet{MaejimaUeda2009e}.
This is why we consider 
nested classes of $L^{\la\alpha\ra}(\rd)$ based on not only $I(\rd)$ but also 
general c.c.s.s.\ $H\subset I(\rd)$.
\end{rem}
Let $\mu\in L_\infty(\rd)=L^{\la0\ra}_\infty(\rd)$, 
and let $\Gamma$ and $\lambda_\beta$ be the measures in Theorem \ref{L^<alpha>_infty} with $\alpha=0$.
We call $\Gm$ the $\Gm$-measure of $\mu\in L_{\infty}(\R^d)$, sometimes
denoted by $\Gm^{\mu}$. We also  write $\ld_{\bt}^\mu$ for $\ld_{\bt}$.
For a set $A\in \B((0,2))$,
let $L_{\infty}^{A}(\R^d)$
denote the class of $\mu\in L_{\infty}(\R^d)$ with $\Gm^{\mu}$
satisfying $\Gm^{\mu}\left((0,2)\setminus A\right)=0$.
Note that $L^{\la\alpha\ra}_\infty(\rd)=L^{[\alpha,2)}_\infty(\rd)$ for $\alpha\in(0,2)$ due to Theorem \ref{L^<alpha>_infty}.
\begin{lem}\label{L^{<alpha>}_infty cap C_alpha}
Let $0<\al <2$.
We have
$L^{\la\alpha\ra}_\infty(\rd)\cap \mathcal C_\alpha(\rd)=L^{(\alpha,2)}_\infty(\rd)$.
\end{lem}
\begin{proof}
Let $\mu\in L^{\la\alpha\ra}_\infty(\rd)=L^{[\alpha,2)}_\infty(\rd)$ with L\'evy measure $\nu$. Then, 
$$
\nu(B)=\int_{[\alpha,2)}\Gamma^\mu(d\beta)\int_S\lambda_\beta^\mu(d\xi)\int_0^\infty\1_B(r\xi)r^{-\beta-1}dr,
\quad B\in\B(\rd).
$$
Since
$$
r^\alpha\int_{|x|>r}\nu(dx)
=r^\alpha\int_{[\alpha,2)}\Gamma^\mu(d\beta)\int_r^\infty u^{-\beta-1}du
=\int_{[\alpha,2)}\beta^{-1}r^{\alpha-\beta}\Gamma^\mu(d\beta),
$$
it follows from the bounded convergence theorem that 
$$
\lim_{r\to\infty}r^\alpha\int_{|x|>r}\nu(dx)=\alpha^{-1}\Gamma^\mu(\{\alpha\}).
$$
Thus $\mu\in \mathcal C_\alpha(\rd)$ if and only if $\Gamma^\mu(\{\alpha\})=0$ under the condition 
$\mu\in L^{\la\alpha\ra}_\infty(\rd)=L^{[\alpha,2)}_\infty(\rd)$.
\end{proof}
Using the lemma above, we have the following.
\begin{thm}\label{Phi_alpha^infty(H) 2}
Let $H\subset I(\rd)$ be c.c.s.s.
\begin{enumerate}[(i)]
\item If $0<\alpha<1$ and $H\supset \bigcup_{\beta\in[\alpha,2]}S_\beta(\rd)$, then 
$$
\lim_{m\to\infty}\Phi_{\alpha}^m\left(H\cap \mathfrak D(\Phi_{\alpha}^m)\right) = L^{(\alpha,2)}_\infty(\rd).
$$
\item If $\alpha=1$ and $H\supset \bigcup_{\beta\in[1,2]}S_\beta(\rd)$, then 
\begin{align*}
\lim_{m\to\infty}\Phi_1^m&\left(H\cap \mathfrak D(\Phi_1^m)\right) = 
\Biggl\{\mu=\mu_{(A,\nu,\gamma)}\in L^{(1,2)}_\infty(\rd)\colon\\ 
&\lim_{\varepsilon\downarrow 0}\int_{(1,2)}B\left(\frac{3-\beta}{2},\frac{\beta+1}{2}\right)
\frac{1-\varepsilon^{\beta-1}}{\beta-1}\Gamma^\mu(d\beta)\int_S \xi\lambda_\beta^\mu(d\xi)=-\gamma
\Biggr\}.
\end{align*}
\item If $1<\alpha<2$ and $H\supset \bigcup_{\beta\in[\alpha,2]}S_\beta(\rd)$, then 
$$
\lim_{m\to\infty}\Phi_{\alpha}^m\left(H\cap \mathfrak D(\Phi_{\alpha}^m)\right) = L^{(\alpha,2)}_\infty(\rd)\cap I_1^0(\rd).
$$
\end{enumerate}
\end{thm}
\begin{proof}
The statements (i) and (iii) come from Corollary \ref{Phi_alpha^infty(H)} and Lemma \ref{L^{<alpha>}_infty cap C_alpha}.

Let us prove the statement (ii).
Suppose $\mu\in L^{(1,2)}_\infty(\rd)$ with L\'evy measure $\nu$. Let $c:=\int_\rd(|x|^2\wedge 1)\nu(dx)$.
Note that $L^{(1,2)}_\infty(\rd)\subset L^{[1,2)}_\infty(\rd)=L^{\la1\ra}_\infty(\rd)\subset L^{\la1\ra}(\rd)$.
Let $\lambda$ and $k_\xi(r)$ be the ones in Theorem \ref{characterization radial} with $\alpha=1$.
It follows from Theorem \ref{monotone h-function} and Lemma \ref{Gamma_xi} with $\alpha=1$ that
$$
k_\xi(r)=\int_{[1,2)}r^{1-\beta}\Gamma_\xi(d\beta),\quad \int_{[1,2)}\left(\frac 1\beta+
\frac 1{2-\beta}\right)\Gamma_\xi(d\beta)=c,
$$
where $\Gamma_\xi,\xi\in S$, are the measures in Lemma \ref{Gamma_xi} with $\alpha=1$. 
Choosing a $[1,2)$-valued random variable $X$ and an $S$-valued random variable $Y$ with joint distribution
$$
P\left(X\in d\beta,\ Y\in d\xi\right):=c^{-1}\left(\frac 1\beta+\frac 1{2-\beta}\right)\lambda(d\xi)\Gamma_\xi(d\beta),
$$
we have
$$
\Gamma^\mu(d\beta)=c\left(\frac 1\beta+\frac 1{2-\beta}\right)^{-1}P(X\in d\beta),\quad 
\lambda_\beta^\mu(d\xi)=P(Y\in d\xi\mid X=\beta)\ \ \Gm^{\mu}\text{-a.e. }\beta
$$
from the uniqueness of $\Gamma^\mu$ and $\lambda_\beta^\mu$. 
Since $\Gamma^\mu(\{1\})=0$, it follows that $\Gamma_\xi(\{1\})=0$ $\lambda$-a.e. $\xi\in S$.
Then we have
\begin{align*}
-\int_\varepsilon^1&t dt \int_S\xi\lambda(d\xi)\int_0^\infty 
\frac{r^2}{1+t^2r^2}dk_\xi(r)\\
&=-\int_\varepsilon^1t dt \int_S\xi\lambda(d\xi)
\int_{(1,2)}\Gamma_\xi(d\beta)\int_0^\infty 
\frac{r^2}{1+t^2r^2}dr^{1-\beta}\\
&=\int_\varepsilon^1t dt \int_S\xi\lambda(d\xi)
\int_{(1,2)}(\beta-1)\Gamma_\xi(d\beta)\int_0^\infty 
\frac{r^{2-\beta}}{1+t^2r^2}dr\\
&=\int_\varepsilon^1 dt \int_S\xi\lambda(d\xi)
\int_{(1,2)}(\beta-1)t^{\beta-2}\Gamma_\xi(d\beta)\int_0^\infty 
\frac{s^{2-\beta}}{1+s^2}ds,
\end{align*}
which is, by 3.251.2 in \citet{GradshteynRyzhik2007}, equal to
\begin{align*}
\int_\varepsilon^1& dt \int_S\xi\lambda(d\xi)
\int_{(1,2)}\frac{\beta-1}2 B\left(\frac{3-\beta}{2},\frac{\beta-1}{2}\right)t^{\beta-2}\Gamma_\xi(d\beta)\\
&=\int_\varepsilon^1 dt \int_S\xi\lambda(d\xi)
\int_{(1,2)}B\left(\frac{3-\beta}{2},\frac{\beta+1}{2}\right)t^{\beta-2}\Gamma_\xi(d\beta)\\
&=\int_\varepsilon^1 dt 
\int_{(1,2)}B\left(\frac{3-\beta}{2},\frac{\beta+1}{2}\right)t^{\beta-2}\Gamma^\mu(d\beta)\int_S\xi\lambda_\beta^\mu(d\xi)\\
&=
\int_{(1,2)}B\left(\frac{3-\beta}{2},\frac{\beta+1}{2}\right)\frac{1-\varepsilon^{\beta-1}}{\beta-1}
\Gamma^\mu(d\beta)\int_S\xi\lambda_\beta^\mu(d\xi).
\end{align*}
Thus, under the condition $\mu\in L^{(1,2)}_\infty(\rd)$, 
$$
\lim_{\varepsilon\downarrow 0}\int_\varepsilon^1t dt \int_S\xi\lambda(d\xi)\int_0^\infty 
\frac{r^2}{1+t^2r^2}dk_\xi(r)=\gamma
$$
if and only if
$$
\lim_{\varepsilon\downarrow 0}\int_{(1,2)}B\left(\frac{3-\beta}{2},\frac{\beta+1}{2}\right)\frac{1-\varepsilon^{\beta-1}}{\beta-1}\Gamma^\mu(d\beta)\int_S\xi\lambda_\beta^\mu(d\xi)=-\gamma.
$$
This completes the proof due to Corollary \ref{Phi_alpha^infty(H)} and Lemma \ref{L^{<alpha>}_infty cap C_alpha}.
\end{proof}

Letting $H=I(\rd)$ in Theorems \ref{Phi_alpha^m(H)} and \ref{Phi_alpha^infty(H) 2} and Corollary \ref{Phi_alpha^infty(H)}, we have the following two theorems.
\begin{thm}\label{R(Phi_{alpha}^infty)}
Let $m\in\Z_+$.
\begin{enumerate}[(i)]
\item When $\alpha \leq 0$,  $\mathfrak R(\Phi_{\alpha}^{m+1}) = L^{\la\alpha\ra}_m(\rd)$.
\item When $ 0<\alpha<1$, $\mathfrak R(\Phi_{\alpha}^{m+1}) = L^{\la\alpha\ra}_m(\rd)\cap \mathcal C_\alpha(\rd)$.
\item When $\alpha =1$,  $\mathfrak R(\Phi_1^{m+1}) = L^{\la1\ra}_m(\rd)\cap \mathcal C_1^*(\rd)$.
\item When $1<\alpha<2$, $\mathfrak R(\Phi_{\alpha}^{m+1}) = L^{\la\alpha\ra}_m(\rd)\cap \mathcal C_\alpha^0(\rd)$.
\end{enumerate}
\end{thm}
\begin{thm}\label{Phi_alpha^infty}
\begin{enumerate}[(i)]
\item When $\alpha \leq 0$,  
$$
\lim_{m\to\infty}\mathfrak R(\Phi_{\alpha}^{m+1}) = L_\infty(\rd).
$$
\item When $ 0<\alpha<1$, 
$$
\lim_{m\to\infty}\mathfrak R(\Phi_{\alpha}^{m+1}) = L^{(\alpha,2)}_\infty(\rd).
$$
\item When $\alpha =1$,  
\begin{align*}
\lim_{m\to\infty}\mathfrak R(\Phi_1^{m+1}) = 
\Biggl\{&\mu=\mu_{(A,\nu,\gamma)}\in L^{(1,2)}_\infty(\rd)\colon\\ 
&\lim_{\varepsilon\downarrow 0}\int_{(1,2)}B\left(\frac{3-\beta}{2},\frac{\beta+1}{2}\right)
\frac{1-\varepsilon^{\beta-1}}{\beta-1}\Gamma^\mu(d\beta)\int_S \xi\lambda_\beta^\mu(d\xi)=-\gamma
\Biggr\}.
\end{align*}
\item When $1<\alpha<2$, 
$$
\lim_{m\to\infty}\mathfrak R(\Phi_{\alpha}^{m+1}) = L^{(\alpha,2)}_\infty(\rd)\cap I_1^0(\rd).
$$
\end{enumerate}
\end{thm}
\begin{rem}
The two theorems above in the case $\alpha=0$ are well-known results.
Also, Theorem \ref{Phi_alpha^infty} in the case $-1\leq \alpha<0$ is already proved in Example 3.5 (5) of \citet{MaejimaSato2009}.
Mappings having the same iterated limits as those of $\Phi_{\alpha},\alpha\in(0,2),$ were already found by \citet{Sato20072008}.
\end{rem}
\vskip 10mm
\section{A supplementary remark}
Theorems \ref{L infty} and \ref{Phi_alpha^infty} have given us the limits of the nested subclasses in terms of limit theorems and mappings, respectively,
where, the forms of the limits look quite dependent on $\alpha$.
However, if we do not care explicit forms of the classes, we can unify the expressions of the results
into one expression as follows.
The first one is a restatement of Theorem \ref{L infty}.
\begin{thm}\label{L infty=L cap S}
Let $\alpha\in\R$.
Then
$L^{\la\alpha\ra}_\infty(\rd)=L^{\la\alpha\ra}(\rd)\cap\overline{S(\rd)}$.
\end{thm}
\begin{proof}
(iii) and (iv) of Proposition \ref{supset} assure the statement for $\alpha\geq 2$.
If $\alpha\leq 0$, Proposition \ref{supset} (i) and Theorem \ref{L infty} yields the statement.
Let $0<\alpha<2$. Then Propositions \ref{The properties of L^{<alpha>}_m} (iii) and \ref{L infty} (ii) yields that $L^{\la\alpha\ra}_\infty(\rd)\subset L^{\la\alpha\ra}(\rd)\cap\overline{S(\rd)}$.
Let $\mu\in L^{\la\alpha\ra}(\rd)\cap\overline{S(\rd)}$ with L\'evy measure $\nu$. Since $\mu\in\overline{S(\rd)}$, we have
$$
\nu(B)=\int_{(0,2)}\Gamma^\mu(d\beta)\int_S\lambda_\beta^\mu(d\xi)\int_0^\infty\1_B(r\xi)r^{-\beta-1}dr,
\quad B\in\B(\rd).
$$
Since $\mu\in L^{\la\alpha\ra}(\rd)$, we have that for all $\alpha'\in(0,\alpha)$, $\int_{|x|>1}|x|^{\alpha'}\nu(dx)<\infty$. Then
$$
\int_{(0,2)}\Gamma^\mu(d\beta)\int_1^\infty r^{\alpha'-\beta-1}dr<\infty,
$$
which entails $\Gamma^\mu((0,\alpha'])=0$ for all $\alpha'\in(0,\alpha)$.
Therefore $\Gamma^\mu((0,\alpha))=0$.
It follows from Theorem \ref{L^<alpha>_infty} (ii) that $\mu\in L^{\la\alpha\ra}_\infty(\rd)$.
\end{proof}
Using the theorem above, we have the following, which is a restatement  of Theorem \ref{Phi_alpha^infty}.
\begin{thm}
Let $\alpha\in\R$.
Then
\begin{equation}\label{limit=range cap S}
\lim_{m\to\infty}\mathfrak R(\Phi_{\alpha}^{m+1}) =\mathfrak R(\Phi_{\alpha})\cap\overline{S(\rd)}.
\end{equation}
\end{thm}
\begin{proof}
If $\alpha\geq 2$, then $\lim_{m\to\infty}\mathfrak R(\Phi_{\alpha}^{m+1}) =\mathfrak R(\Phi_{\alpha})\cap\overline{S(\rd)}=\{\delta_0\}$.
Combining Theorem \ref{L infty=L cap S} with \eqref{Phi range} and Theorem \ref{R(Phi_{alpha}^infty)}, we have the statement for $\alpha<2$.
\end{proof}
\begin{rem}
Many known mappings satisfies \eqref{limit=range cap S}, (see, e.g. \citet{MaejimaSato2009}).
However, some mappings does not fulfill \eqref{limit=range cap S}, (see, e.g. \citet{MaejimaUeda2009}).
\end{rem}
\vskip 10mm

\vskip 10mm
\end{document}